\documentclass[a4paper]{siamart171218}

\usepackage{mathrsfs}
\usepackage{mathdots}
\usepackage{amsmath,amssymb}
\usepackage[utf8]{inputenc}
\usepackage{enumitem}
\usepackage{algorithm}
\usepackage{algpseudocode}
\usepackage{pgfplotstable}
\usepackage[group-separator={,}]{siunitx}
\usepackage{cite}
\usepackage{url}
\usepackage{booktabs}

\DeclareMathOperator{\fd}{FD}
\DeclareMathOperator{\fe}{FE}

\newcommand{\fdm}[1][]{%
	\ifthenelse{\equal{#1}{}}{%
		\mathcal{M}^{\fd,m}_{\alpha,N}%
		}{%
		\mathcal{M}^{\fd,m}_{#1,N}%
		}%
}
\newcommand{\fds}[1][]{%
	\ifthenelse{\equal{#1}{}}{%
		\mathcal{M}^{\fd}_{\alpha,N}%
	}{%
	\mathcal{M}^{\fd}_{#1,N}%
}%
}
\newcommand{\fem}[1][]{%
	\ifthenelse{\equal{#1}{}}{%
		\mathcal{M}^{\fe,m}_{\alpha,N}%
	}{%
	\mathcal{M}^{\fe,m}_{#1,N}%
}%
}
\newcommand{\fems}[1][]{%
\ifthenelse{\equal{#1}{}}{%
	\mathcal{M}^{\fe}_{\alpha,N}%
}{%
\mathcal{M}^{\fe}_{#1,N}%
}%
}
\newcommand{\femsp}{\mathcal{M}^{\fe,+}_{\alpha,N}}
\newcommand{\femsm}{\mathcal{M}^{\fe,-}_{\alpha,N}}

\newcommand\TheTitle{Fast solvers for two-dimensional
	fractional diffusion equations using
	rank structured matrices}
\newcommand\TheShortTitle{Fast solvers for 2D
	FDEs using
	rank structures}
\newcommand\TheAuthors{Stefano Massei, Mariarosa Mazza, Leonardo
	Robol} 
 \newcommand{\vect} {\mathrm{vec}}
\newcommand{\supp}{\mathrm{supp}} 
\headers{\TheShortTitle}{\TheAuthors}

\renewcommand{\le}{\leq}
\renewcommand{\ge}{\geq}
\newcommand{\norm}[1]{\lVert #1 \rVert}
\definecolor{lowrankcolor}{rgb}{.75,.75,.75}

\usepackage{geometry}

\title{\TheTitle\thanks{%
	This work has been partially supported by an INdAM/GNCS
	project. The three authors are members of the research group GNCS. The work of the first author has been supported by the SNSF research project Fast algorithms from low-rank updates, grant number: 200020 178806. The work of the third author has been supported by the Region of Tuscany (PAR-FAS 2007 -- 2013) and by MIUR,
	the Italian Ministry of Education, Universities and Research (FAR) within the Call FAR -- FAS 2014
	(MOSCARDO Project: ICT technologies for structural monitoring of age-old constructions based
	on wireless sensor networks and drones, 2016 -- 2018).}}

\author{%
	Stefano Massei\thanks{EPF Lausanne, Switzerland,
		\email{stefano.massei@epfl.ch}} \and Mariarosa Mazza\thanks{Department of Science and High Technology, University of Insubria, Como, Italy, and Max Planck Institute for Plasma Physics, Munich, Germany, 
		\email{mariarosa.mazza@uninsubria.it}} \and Leonardo
	Robol\thanks{Department of Mathematics, University of Pisa, and 
		 ISTI-CNR, Pisa, Italy, \email{leonardo.robol@unipi.it}}}

\pgfplotstableset{
	every head row/.style={before row=\toprule,after row=\midrule},
	clear infinite
}

\numberwithin{theorem}{section}
\newsiamremark{remark}{Remark}
\newsiamremark{example}{Example}

\renewcommand{\leq}{\leqslant}
\renewcommand{\geq}{\geqslant}
\renewcommand{\tilde}{\widetilde}

\ifpdf \hypersetup{ pdftitle={\TheTitle},
	pdfauthor={\TheAuthors} } \fi

\begin{document}
	\maketitle
	\begin{abstract}
		We consider the  discretization of time-space diffusion equations with fractional derivatives in space and either one-dimensional (1D) or 2D spatial domains. The use of an implicit Euler scheme in time and finite differences or finite elements in space leads to a sequence of dense large scale linear systems describing the behavior of the solution over a time interval.		
		We prove that the coefficient matrices arising in the 1D context are rank structured and can be efficiently represented using hierarchical formats ($\mathcal H$-matrices, HODLR). Quantitative estimates for the rank of the
		off-diagonal blocks of these matrices are presented. 		
		We analyze the use of HODLR arithmetic
		for solving the 1D case and we compare this strategy with
		existing methods that exploit the Toeplitz-like structure to precondition the GMRES iteration. The numerical tests demonstrate the convenience of the HODLR format when at least a reasonably low number of time steps is needed.
Finally, we explain how these properties can be leveraged to
		design fast solvers for problems with 2D spatial domains that can be reformulated as matrix equations. The experiments show that the approach
		based on the use of rank-structured arithmetic is particularly effective and outperforms current
		state of the art techniques.
	\end{abstract}



\section{Introduction}
	
	Fractional Diffusion Equations (FDEs) are a generalization of the classical partial diffusion equations obtained by replacing a standard derivative with a fractional one.
	In the last decade, FDEs have gained a lot of attention since they allow to model non-local behavior, e.g., enhanced diffusivity, which can be regarded as a realistic representation of specific physical phenomena appearing  in several applications.  In finance, this
	is used to  take long time correlations into consideration \cite{fin};
	in image processing, the use of fractional
	anisotropic diffusion allows to accurately
	recover images from their corrupted or noisy version --- without incurring
	the risks of over-regularizing the solution and thus losing significant part
	of the image such as the edges \cite{imag}. 	
	The applications in fusion plasma physics concern Tokamak reactors (like ITER currently under construction in the South of France \cite{toka})	which are magnetic toroidal confinement devices aiming to harvest energy from the fusion of small atomic nuclei, typically Deuterium and Tritium,  heated to the plasma state. Recent experimental and theoretical evidence indicates that transport in Tokamak reactors deviates from the standard diffusion paradigm. One of the proposed models that incorporates in a natural, unified way, the unexpected anomalous diffusion phenomena is based on the use of fractional derivative operators \cite{plasma}. The development of fast numerical tools for solving the resulting equations is then a key requirement for controlled thermonuclear fusion, which offers the possibility of clean, sustainable, and almost limitless energy.
	
	Let us briefly recall how a standard diffusion equation can be ``fractionalized'', by taking as an example the parabolic diffusion
	equation $\frac{\partial u(x,t)}{\partial t}=~d(x,t)\frac{\partial^2 u(x,t)}{\partial x^2}$, where $d(x,t)$ is the diffusion coefficient.
	Replacing the derivative in time with a fractional one leads to
	a \emph{time-fractional} diffusion equation; in this case, the
	fractional derivative order is chosen between $0$ and $1$. On the
	other hand, we can consider a \emph{space-fractional} diffusion equation
	by introducing a fractional derivative in space, with
	order between $1$ and $2$. The two approaches (which can also be combined),
	lead to similar computational issues. In this paper, we focus on the space fractional initial-boundary value problem
		\begin{eqnarray}\label{fde_FD}
		\left\{
		\begin{array}{lc}
		\frac{\partial u(x,t)}{\partial t}=d_{+}(x,t)\frac{\partial^\alpha u(x,t)}{\partial_{+} x^\alpha}+d_{-}(x,t)\frac{\partial^\alpha u(x,t)}{\partial_{-} x^\alpha}+f(x,t), & (x,t)\in(L,R)\times(0,T], \\
		u(x,t)=u(x,t)=0, & (x,t)\in\mathbb{R}\backslash(L,R)\times[0,T],\\
		u(x,0)=u_{0}(x), & x\in[L,R],
		\end{array}
		\right.
		\end{eqnarray}
		where $\alpha\in(1,2)$ is the fractional derivative order, $f(x,t)$ is the \emph{source term}, and the nonnegative functions $d_{\pm}(x,t)$ are the \emph{diffusion coefficients}. Three of the most famous definitions of the right-handed (--) and the left-handed (+) fractional derivatives in \eqref{fde_FD} are due to Riemann–Liouville, Caputo, and Gr\"unwald-Letnikov. We refer the reader to
		Section~\ref{sec:definitions} for a detailed discussion.

As suggested in \cite{deng}, in order to guarantee the well-posedness of a space-FDE problem, the value of the solution on $\mathbb{R}\backslash(L,R)$ must be properly accounted for. In this view, in \eqref{fde_FD} we fix the so-called \emph{ absorbing boundary conditions}, that is we assume that the particles are ``killed'' whenever they leave the domain $(L,R)$.

		Analogously, one can generalize high dimensional differential operators by applying (possibly different order) fractional derivatives in each coordinate direction. More explicitly, we consider the 2D extension of \eqref{fde_FD}
\begin{align}\label{fde_FD2d}
\begin{split}
		\frac{\partial u(x,y,,t)}{\partial t} &= d_{1,+}(x,t)\frac{\partial^{\alpha_1} u(x,y,t)}{\partial_+ x^{\alpha_1}} +d_{1,-}(x,t)\frac{\partial^{\alpha_1} u(x,y,t)}{\partial_- x^{\alpha_1}} \\
&+ d_{2,+}(y,t)\frac{\partial^{\alpha_2} u(x,y,t)}{\partial_+ y^{\alpha_2}} +d_{2,-}(y,t)\frac{\partial^{\alpha_2} u(x,y,t)}{\partial_- y^{\alpha_2}}+f(x,y,t).
\end{split}
\end{align}
with absorbing boundary conditions. Notice that the diffusion coefficients only depend on time and on the variable of the corresponding differential operator. This choice makes it easier to treat a 2D FDE problem as a matrix equation and to design fast numerical procedures for it (see Section \ref{sec:matrix-eq}).

\subsection{Existing numerical methods for FDE problems}
The non-local nature of fractional differential operators causes the absence of sparsity in the
coefficient matrix of the corresponding discretized problem. This makes FDEs computationally more demanding than PDEs. 		

Various numerical discretization methods for FDE problems, e.g., finite differences, finite volumes, finite elements have been the subject of many studies \cite{Meer1,TZD,ervin}. In the case of regular spatial domains, the discretization matrices often inherit a Toeplitz-like structure from the space-invariant property of the underlying operators. Iterative schemes such as  multigrid and preconditioned Krylov methods --- able to exploit this structure --- can be found in \cite{wang,LS,PS,M1,ng2016SCI,ng2016,DMS,2D}.
		For both one- and two-dimensional FDE problems, we mention the structure preserving preconditioning and the algebraic multigrid methods presented in \cite{DMS,2D}. Both strategies are based on the spectral analysis of the coefficient matrices via their spectral symbol. The latter is a function which provides a compact spectral description of the discretization matrices whose computation relies on the theory of Generalized Locally Toeplitz (GLT) matrix-sequences \cite{GS}.
		
		Only very recently, off-diagonal rank structures have been recognized in finite element discretizations \cite{zhao}. Indeed, Zhao et al. proposed the use of hierarchical matrices for storing the stiffness matrix combined with geometric multigrid (GMG) for solving the linear system.

		It is often the case that 2D problems with piecewise smooth right-hand sides have piecewise smooth solutions (see, e.g., \cite{caffarelli}). A possible way to uncover and  leverage this property is to rephrase the linear system in matrix equation form.
	This is done for instance in \cite{Breiten2014}, where the authors use the Toeplitz-like structure of the involved one-dimensional FDE matrices in combination with the extended Krylov subspace method to deal with fine discretizations.
	
	\subsection{Motivation and contribution}
		 In this paper, we aim to build a framework for
		analyzing the
		low-rank properties of one-dimensional FDE discretizations and
		to design fast algorithms for solving two-dimensional FDEs written in a matrix
		equation form. In detail, the piecewise smooth property of the right-hand side implies the low-rank structure in the solution of the matrix equation and enables the use of Krylov subspace methods \cite{Breiten2014}.
		This, combined with the technology of hierarchically rank-structured matrices such as
		$\mathcal{H}$-matrices and HODLR \cite{hackbusch2015hierarchical}, yields a linear-polylogarithmic computational
		 complexity in the size of the edge of the mesh.
		For instance, for a $N\times N$ grid we get a linear
		polylogarithmic cost in $N$, in contrast to $\mathcal
		O(N^2\log N)$ needed by a multigrid approach or a
		preconditioned iterative method applied to linear
		systems with dense Toeplitz coefficient matrices. Similarly, the storage consumption is reduced
		from $\mathcal O(N^2)$ to $\mathcal O(N\log N)$.
		The numerical experiments demonstrate that our approach,
		based on the HODLR format, outperforms the one proposed in
		\cite{Breiten2014}, although the asymptotic cost is comparable.
		From the theoretical side, we provide an analysis of the rank structure
		in the matrices coming from the discretization of fractional differential
		operators. Our main results claim that the off-diagonal blocks in these
		matrices have numerical rank $\mathcal O(\log(\epsilon^{-1})\log(N))$ where $\epsilon$ is the truncation threshold. We highlight that some of these
		results do not rely on the Toeplitz structure and apply to more general
		cases, e.g., stiffness matrices of finite element methods on non-uniform meshes.
		
		The use of hierarchical matrices for finite element discretization of FDEs has already been
		explored in the literature. For instance,
		the point of view in \cite{zhao} is similar to the
		one we take in Section~\ref{sec:anal2}; however,
		we wish to highlight two main differences with our
		contribution. First, Zhao et al. considered the adaptive
		geometrically balanced clustering \cite{geom-clus}, in place of the HODLR partitioning. 
As mentioned in the Section~\ref{sec:hodlr-fe}, there is only little difference between the off-diagonal ranks of the two partitionings, 
hence HODLR arithmetic turns out to be preferable because it reduces the storage consumption. 
		Second, they propose the use of geometric multigrid (GMG) for solving linear systems with the stiffness matrix. In the case of multiple time steps and for generating the extended Krylov subspace, precomputing the LU factorization is more convenient, as we discuss in Section~\ref{sec:exp}. To the best of our knowledge, the rank
		structure in finite difference discretizations of fractional
		differential operators has not been previously noticed.
		
		The paper is organized as follows; in Section~\ref{sec:definitions} we recall different definitions of fractional derivatives and the discretizations proposed in the literature. Section~\ref{sec:rank} is dedicated to the study of the rank structure arising in the discretizations. More specifically, in Section~\ref{sec:glt} a preliminary qualitative analysis is supported by the GLT theory, providing a decomposition of each off-diagonal block as the sum of a low-rank plus a small-norm term. Then, a quantitative analysis is performed in Sections~\ref{sec:anal1} and \ref{sec:anal2}, with different techniques stemming from the framework of structured matrices. In Section~\ref{sec:practice} we introduce the HODLR format and we discuss how to efficiently construct representations of the matrices of interest. Section~\ref{sec:solv-1D} briefly explains how to combine these ingredients to solve the 1D problem. In Section~\ref{sec:matrix-eq} we reformulate the 2D problem as a Sylvester matrix equation with structured coefficients and we illustrate a fast solver. The performances of our approach are tested and compared in Section~\ref{sec:exp}. Conclusion and future outlook are given in Section~\ref{sec:concl}.
			
		\section{Fractional derivatives and their discretizations}\label{sec:definitions}
		
		\subsection{Definitions of fractional derivatives}
		
		A common definition of fractional derivatives is given by the Riemann–Liouville formula. For a given function with absolutely continuous first derivative on $[L,R]$, the right-handed and left-handed Riemann–Liouville fractional derivatives of order $\alpha$ are defined by
		\begin{equation}\label{eq:frac-der}
		\begin{split}
		\frac{\partial^\alpha u(x,t)}{\partial^{RL}_{+} x^\alpha}&= \frac{1}{\Gamma(n-\alpha)}\frac{\partial^n}{\partial x^n}\int_L^x\frac{u(\xi,t)}{(x-\xi)^{\alpha+1-n}}d\xi,\\
		\frac{\partial^\alpha u(x,t)}{\partial^{RL}_{-} x^\alpha}&=\frac{(-1)^n}{\Gamma(n-\alpha)}\frac{\partial^n}{\partial x^n}\int_x^R\frac{u(\xi,t)}{(\xi-x)^{\alpha+1-n}}d\xi,
		\end{split} \end{equation}
		where $n$ is the integer such that $n-1<\alpha\le n$ and $\Gamma(\cdot)$ is the Euler gamma function. Note that the left–handed fractional derivative of the function $u(x,t)$ computed at $x$ depends on all function
		values to the left of  $x$, while the right–handed fractional derivative depends on the ones to the right.
		
		When $\alpha=m$, with $m\in\mathbb{N}$, then \eqref{eq:frac-der} reduces to the standard integer derivatives, i.e.,
		\begin{eqnarray*}
			\frac{\partial^m u(x,t)}{\partial^{RL}_{+} x^m}=\frac{\partial^m u(x,t)}{\partial x^m}, \quad \frac{\partial^m u(x,t)}{\partial^{RL}_{-} x^m}=(-1)^m\frac{\partial^m u(x,t)}{\partial x^m}.
		\end{eqnarray*}
		An alternative definition is based on the
		Gr\"unwald--Letnikov formulas:	
		\begin{align}\label{Grun}
			\begin{split}
				\frac{\partial^\alpha u(x,t)}{\partial^{GL}_{+} x^\alpha}&=\lim_{\Delta x\rightarrow0^{+}}\frac{1}{\Delta x^\alpha}\sum_{k=0}^{\lfloor(x-L)/\Delta x\rfloor}g_k^{(\alpha)}u(x-k\Delta x,t),\\
				\frac{\partial^\alpha u(x,t)}{\partial^{GL}_{-} x^\alpha}&=\lim_{\Delta x\rightarrow0^{+}}\frac{1}{\Delta x^\alpha}\sum_{k=0}^{\lfloor(R-x)/\Delta x\rfloor}g_k^{(\alpha)}u(x+k\Delta x,t),
			\end{split}
		\end{align}
		where $\lfloor\cdot\rfloor$ is the floor function, $g_k^{(\alpha)}$ are the alternating fractional binomial coefficients
		\begin{align*}
			g_k^{(\alpha)}=(-1)^k
			  \binom{\alpha}{k}
			  =\frac{(-1)^k}{k!}\alpha(\alpha-1)\cdots(\alpha-k+1)\qquad k=1,2,\ldots
		\end{align*}
		and $g_0^{(\alpha)} = 1$. Formula \eqref{Grun} can be seen as an extension of the definition of ordinary derivatives via limit of the difference quotient.

Finally, another common definition of fractional derivative was proposed by Caputo:
\begin{equation}\label{eq:caputo}
\begin{split}
\frac{\partial^\alpha u(x,t)}{\partial^{C}_{+} x^\alpha}&= \frac{1}{\Gamma(n-\alpha)}\int_L^x\frac{\frac{\partial^n}{\partial \xi^n}u(\xi,t)}{(x-\xi)^{\alpha+1-n}}d\xi,\\
\frac{\partial^\alpha u(x,t)}{\partial^{C}_{-} x^\alpha}&=\frac{(-1)^n}{\Gamma(n-\alpha)}\int_x^R\frac{\frac{\partial^n}{\partial \xi^n}u(\xi,t)}{(\xi-x)^{\alpha+1-n}}d\xi.
\end{split} \end{equation}

Note that \eqref{eq:caputo} requires the $n$th derivative of $u(x,t)$ to be absolutely integrable. Higher regularity of the solution is typically imposed in time rather than in space; as a consequence, the Caputo formulation is mainly used for fractional derivatives in time, while Riemann--Liouville's is preferred for fractional derivatives in space. The use of Caputo's derivative provides some advantages in the treatment of boundary conditions when applying the Laplace transform method (see \cite[Chapter 2.8]{frac}).

	The various definitions are equivalent only if $u(x,t)$ is sufficiently regular and/or vanishes with all its derivatives on the boundary. In detail, it holds that:
	
	\begin{itemize}
		\item if the $n$th space derivative of $u(x,t)$ is continuous on $[L,R]$, then
		$$\frac{\partial^\alpha u(x,t)}{\partial^{RL}_{+} x^\alpha}=\frac{\partial^\alpha u(x,t)}{\partial^{GL}_{+} x^\alpha}, \qquad \frac{\partial^\alpha u(x,t)}{\partial^{RL}_{-} x^\alpha}=\frac{\partial^\alpha u(x,t)}{\partial^{GL}_{-} x^\alpha}.$$
		\item if $\frac{\partial^\ell}{\partial x^\ell}u(L,t)=\frac{\partial^\ell}{\partial x^\ell}u(R,t)=0$ for all $\ell=0,1,\dots,n-1$, then $$\frac{\partial^\alpha u(x,t)}{\partial^{RL}_{+} x^\alpha}=\frac{\partial^\alpha u(x,t)}{\partial^{C}_{+} x^\alpha}, \qquad \frac{\partial^\alpha u(x,t)}{\partial^{RL}_{-} x^\alpha}=\frac{\partial^\alpha u(x,t)}{\partial^{C}_{-} x^\alpha};$$
	\end{itemize}		
	In this work we are concerned with space fractional
	derivatives so we focus on the Riemann--Liouville and the Gr\"unwald--Letnikov formulations. However, the analysis
	of the structure in the discretizations can be generalized
	with minor adjustments to the Caputo case.
	
		\subsection{Discretizations of fractional derivatives}
		
		We consider two different discretization schemes for
		the FDE problem \eqref{fde_FD}: finite differences and
		finite elements.
		The
		first scheme relies on the Gr\"unwald--Letnikov formulation while the second is derived adopting the Riemann--Liouville
		 definition.
		
		\subsubsection{Finite difference scheme using Gr\"unwald-Letnikov formulas}\label{sec:grunw}
		
		As suggested in \cite{Meer1}, in order to obtain a consistent and unconditionally stable finite difference scheme for \eqref{fde_FD}, we use a shifted version of the Gr\"unwald--Letnikov fractional derivatives obtained replacing $k\Delta x$ with $(k-1)\Delta x$ in \eqref{Grun}.

	Let us fix two positive integers $N,M$, and define the following partition of $[L,R]\times[0,T]$:
	\begin{align}\label{part}
		x_i=L+i\Delta x, \quad \Delta x=\frac{(R-L)}{N+1},  \quad i=0,\ldots,N+1,\\
		\notag t_m=m\Delta t, \quad \Delta t=\frac{T}{M+1}, \quad m=0,\ldots,M+1.&
	\end{align}
	The idea in \cite{Meer1} is to combine a discretization in time of equation \eqref{fde_FD} by an implicit Euler method with a first order discretization in space of the fractional derivatives by a shifted Gr\"unwald-Letnikov estimate, i.e.,
	\begin{equation*}
		\frac{u(x_i,t_m)-u(x_i,t_{m-1})}{\Delta t}=d_{+,i}^{(m)}\frac{\partial^\alpha u(x_i,t_m)}{\partial^{GL}_{+} x^\alpha}+d_{-,i}^{(m)}\frac{\partial^\alpha u(x_i,t_m)}{\partial^{GL}_{-} x^\alpha}+f_i^{(m)}+\mathcal O(\Delta t),
	\end{equation*}
	where $d_{\pm,i}^{(m)}=d_{\pm}(x_i,t_m)$, $f_i^{(m)}:=f(x_i,t_m)$ and
	\begin{align*}
		\frac{\partial^\alpha u(x_i,t_m)}{\partial^{GL}_{+} x^\alpha}&=\frac{1}{\Delta x^\alpha}\sum_{k=0}^{i+1}g_k^{(\alpha)}u(x_{i-k+1},t_m)+\mathcal O(\Delta x),\\
		\frac{\partial^\alpha u(x_i,t_m)}{\partial^{GL}_{-} x^\alpha}&=\frac{1}{\Delta x^\alpha}\sum_{k=0}^{N-i+2}g_k^{(\alpha)}u(x_{i+k-1},t_m)+\mathcal O(\Delta x).
	\end{align*}
	The resulting finite difference approximation scheme is then
	\begin{equation*}
		\frac{u_i^{(m)}-u_i^{(m-1)}}{\Delta t}=\frac{d_{+,i}^{(m)}}{\Delta x^\alpha}\sum_{k=0}^{i+1}g_k^{(\alpha)}u_{i-k+1}^{(m)}+\frac{d_{-,i}^{(m)}}{\Delta x^\alpha}\sum_{k=0}^{N-i+2}g_k^{(\alpha)}u_{i+k-1}^{(m)}+f_i^{(m)},
	\end{equation*}
	where by $u_i^{(m)}$ we denote a numerical approximation of $u(x_i,t_m)$. The previous approximation scheme can be written in matrix form as (see \cite{WWT})
	\begin{equation} \label{system_FD}
	\fdm u^{(m)} =
	\left( I+ \frac{\Delta t}{\Delta x^\alpha}(D_{+}^{(m)}T_{\alpha,N}+D_{-}^{(m)}T_{\alpha,N}^T)\right)u^{(m)}= u^{(m-1)}+\Delta t f^{(m)},
	\end{equation}
	 $u^{(m)}=[u_1^{(m)},\ldots,u_N^{(m)}]^T$, $f^{(m)}=[f_1^{(m)},\ldots,f_N^{(m)}]^T$, $D_{\pm}^{(m)}={\rm diag}(d_{\pm,1}^{(m)},\ldots,d_{\pm,N}^{(m)})$,  $I$ is the identity matrix of order $N$ and
	\begin{equation} \label{toeplitz_FD}
		T_{\alpha,N}=-\left[
		\begin{matrix}
			g_1^{(\alpha)}& g_0^{(\alpha)} & 0 & \cdots & 0 &0\\
			g_2^{(\alpha)}& g_1^{(\alpha)} & g_0^{(\alpha)}& 0 &\cdots & 0\\
			\vdots &\ddots & \ddots & \ddots &\ddots & \vdots\\
			\vdots &\ddots & \ddots & \ddots &\ddots & 0\\
			g_{N-1}^{(\alpha)}&\ddots & \ddots & \ddots & g_1^{(\alpha)}& g_0^{(\alpha)}\\
			g_{N}^{(\alpha)}& g_{N-1}^{(\alpha)} & \cdots & \cdots & g_2^{(\alpha)}& g_1^{(\alpha)}
		\end{matrix}
		\right]_{N\times N}
	\end{equation}
	is a lower Hessenberg Toeplitz matrix.
		Note that $\fdm$ has a Toeplitz-like structure, in the sense that it can be expressed as a sum of products between diagonal and dense Toeplitz matrices. It can be shown that $\fdm$ is strictly diagonally dominant and then non singular (see \cite{Meer1,WWT}), for every choice of the parameters $m\ge 0$, $N\ge 1$, $\alpha \in (1,2)$. Moreover, it holds $g_1^{(\alpha)}=-\alpha$,  $g_0^{(\alpha)}>g_2^{(\alpha)}>g_3^{(\alpha)}>\dots>0$ and  $g_k^{(\alpha)}=\mathcal O(k^{-\alpha-1})$.
	
\subsubsection{Finite element space discretization} \label{sub:FDE2}
  We consider a finite element discretization for \eqref{fde_FD}, using the Riemann--Liouville formulation \eqref{eq:frac-der}. Let $\mathcal B = \{ \varphi_1, \ldots, \varphi_N \}$ be a finite element basis, consisting of positive functions with compact support that vanish on the boundary. At each time step $t$, we replace the true solution $u$ by its finite element approximation $u_{\Delta x}$
	\begin{equation}\label{vfm}
	u_{\Delta x}=\sum_{j=1}^N u_j(t)\varphi_j(x),
	\end{equation}
	then we formulate a finite element scheme for \eqref{fde_FD}. Assuming that the diffusion coefficients do not depend on $t$, by means of this formulation we can find the steady-state solution solving the linear system $\fems u = f$, with
	\[
	(\fems)_{ij} = \langle \varphi_i(x),
	d_{+}(x)\frac{\partial^\alpha \varphi_j(x)}{\partial^{RL}_{+} x^\alpha}+d_{-}(x)\frac{\partial^\alpha \varphi_j(x)}{\partial^{RL}_{-} x^\alpha}\rangle, \qquad
	f_i = \langle -v(x), \varphi_i(x) \rangle.
	\]
By linearity, we can decompose $\fems = \femsp + \femsm$ where $\femsp$ includes the action of the left-handed derivative, and $\femsm$ the effect of the right-handed one.

The original time-dependent equation \eqref{fde_FD} can be solved by means of a suitable time discretization (such as the implicit Euler method used in the previous section) combined with the finite element scheme introduced here (see, e.g., \cite{duan2015finite,liu2017discontinuous}). In the time-dependent case,
the matrix associated with the instant $t_m$ will be denoted by $\fem$, and it has the same structure of $\fems$ up to a shift by the mass matrix $M$:
\[
\fem =   M-\Delta t \fems,\qquad (M)_{ij}=\langle\varphi_i(x),\varphi_j(x)\rangle.
\]
 The resulting time stepping scheme can be expressed as follows
 \[
 \fem u^{(m)} = Mu^{(m-1)}+\Delta t f^{(m)}.
 \]

We refer the reader to \cite{ervin,roop2006computational} for more details on the finite element discretization of fractional problems, including a detailed analysis of the spaces used for the basis functions and convergence properties.

	\section{Rank structure in the 1D case}\label{sec:rank}
	The aim of this section is to prove that different formulations of 1D fractional derivatives generate
	discretizations with similar properties. 	
	In particular, we are interested in showing that off-diagonal blocks
	in the matrix discretization of such operators have a low numerical rank. When these ``off-diagonal ranks'' are exact (and not
	just numerical ranks) this structure is sometimes called
	\emph{quasiseparability}, or \emph{semiseparability} \cite{eidelman:book1,vanbarel:book1}, see also Figure \ref{MML}. Here we recall
	the definition and some basic properties.
	\begin{definition} A matrix $A\in \mathbb C^{N\times N}$ is
		\emph{quasiseparable} of order $k$ (or quasiseparable of rank $k$) if
		the maximum of the ranks of all its submatrices contained in the
		strictly upper or lower part is exactly $k$.
	\end{definition}
	\begin{figure}[!ht] \centering
		\begin{tikzpicture}
		\draw [thick] (0,0) rectangle (2,2); \draw (0,2) -- (2,0);
		
		\fill [lowrankcolor] (0.05,1.25) rectangle
		(0.65,.05); \fill [lowrankcolor] (0.8,1.4) rectangle (1.75,1.75);
		
		\fill [lowrankcolor] (-1.8,1.25) rectangle (-1.65,.05); \fill [lowrankcolor]
		(-1.6,1.25) rectangle (-1.1,1.1); \draw[->] (-1.55,0.6) -- (-.05,0.6);
		
		\fill [lowrankcolor] (2.5,1.75) rectangle (2.65,1.4); \fill
		[lowrankcolor] (2.7,1.75) rectangle (3.2, 1.6); \draw [<-] (1.8,1.575) -- (2.45,1.575);
		\end{tikzpicture}
		\caption{Pictorial description of the quasiseparable structure; the
			off-diagonal blocks can be represented as low-rank outer products.}\label{MML}
	\end{figure}
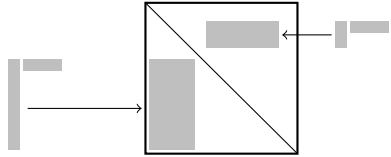
	\begin{lemma} Let $A,B\in\mathbb C^{N\times N}$ be quasiseparable of
		rank $k_A$ and $k_B$, respectively.
		\begin{enumerate}
			\item The quasiseparable rank of both $A+B$ and $A\cdot B$ is at
			most $k_A+k_B$.
			\item If $A$ is invertible then $A^{-1}$ has quasiseparable rank
			$k_A$.
		\end{enumerate}
	\end{lemma} 	
	In order to perform the analysis of the off-diagonal blocks
	we need to formalize the concept of numerical rank. In the rest of the paper, $\norm{\cdot} $ will indicate the Euclidean norm.
	
	\begin{definition}
		We say that a matrix $X$ has $\epsilon$-rank $k$, and we write
		$\mathrm{rank}_\epsilon(X) = k$, if there exists $\delta X$ such that $\norm{\delta X} \leq \epsilon \norm{X}$, $\mathrm{rank}(X + \delta X) = k$ and the rank
		of $X + \delta X'$
		is at least $k$ for any other $\norm{\delta X'} \leq \epsilon \norm{X}$.
		More formally:
		\[
		\mathrm{rank}_\epsilon(X) = \min_{\norm{\delta X} \leq \epsilon \norm{X}} \mathrm{rank}(X + \delta X).
		\]
	\end{definition}
	
	Often we are interested in measuring approximate quasiseparability.
	We can give a similar ``approximate'' definition.
	
	\begin{definition}
		We say that a matrix $X$ has $\epsilon$-qsrank $k$ if for any
		off-diagonal block $Y$ of $X$ there exists a perturbation $\delta Y$ such that
		$\norm{\delta Y} \leq \epsilon \norm{X}$ and $Y + \delta Y$ has
		rank (at most) $k$. More formally:
		\[
		\mathrm{qsrank}_\epsilon(X) = \max_{Y \in \mathrm{Off}(X)} \min_{\norm{\delta Y} \leq \epsilon \norm{X}} \mathrm{rank}(Y + \delta Y),
		\]
		where $\mathrm{Off}(X)$ is the set of the off-diagonal blocks of $X$.
	\end{definition}
	
	\begin{remark}
		As shown in \cite{massei2017solving},  $\mathrm{qsrank}_\epsilon(X)=k$ implies the existence of a ``global" perturbation $\delta X$ such that $X+\delta X$ is quasiseparable of rank $k$ and $\norm{\delta X}\leq \epsilon\sqrt N$. 		
		In addition, the presence of $\norm{X}$ in the above definition makes
		the $\epsilon$-qsrank invariant under rescaling, i.e.,
		$\mathrm{qsrank}_\epsilon(A) = \mathrm{qsrank}_{\epsilon}(\theta A)$
		for any $\theta \in \mathbb C\setminus \{0\}$.
	\end{remark}
	The purpose of the following subsections is to show that the various discretizations of fractional derivatives provide matrices with small $\epsilon$-qsrank. The $\epsilon$-qsrank turns out to grow asymptotically as $\mathcal O(\log(\epsilon^{-1})\log(N))$, see Table~\ref{tab:bounds} which summarizes our findings.
	
	\begin{table}
		\centering
		\begin{tabular}{c|cc}
			Discretization & $\epsilon$-qsrank & Reference\\
			\hline \\[-5pt]
			Finite differences& $2 + 2\left\lceil
			\frac {2}{\pi^2} \log \left( \frac{4}{\pi}N \right)
			\log\left(\frac{32}{\epsilon}\right)
			\right\rceil$ & Lem.~\ref{lem:fd-qsrank}, Cor.~\ref{cor:fd-qsrank} \\[5pt]
			Finite elements & $k +
			2\left\lceil\log_2 \left(\frac{R-L}{\delta} \right)\right\rceil \cdot
			\left(1 + \left\lceil \log_2 \left(\frac{(\alpha+1) \cdot 4^{\alpha+1}}{\epsilon}\right) \right\rceil \right)$ & Thm.~\ref{thm:structure-riesz-fem}\\[5pt]
		\end{tabular}
		\caption{Bounds for the $\epsilon$-qsrank of different discretizations. For finite elements methods with equispaced
		basis functions the parameter $(R-L)/\delta \approx N$,
	    and $k$ is the number of overlapping basis functions (see
	    Definition~\ref{def:overlap}). }
		\label{tab:bounds}
	\end{table}

\subsection{Qualitative analysis of the quasiseparable structure through GLT theory}\label{sec:glt}
In the 1D setting the finite difference discretization matrices $\fdm$ present a diagonal-times-Toeplitz structure --- see \eqref{system_FD} --- where the diagonal matrices are the discrete counterpart of the diffusion coefficients and the Toeplitz components come from the fractional derivatives. This structure falls in the Generalized Locally Toeplitz (GLT) class, an algebra of matrix-sequences obtained as a closure under some algebraic operations (linear combination, product, inversion, conjugation) of Toeplitz, diagonal and low-rank plus small-norm matrix-sequences.

In the remaining part of this section we show that the off-diagonal blocks of $\fdm$ can be decomposed as the sum of a low-rank plus a small-norm term. Such a result is obtained exploiting the properties of  some simple GLT sequences, i.e., Toeplitz and Hankel sequences associated with a function $f\in L^1$.

\begin{definition}
	Let $f\in L^1([-\pi,\pi])$ and let $\left\{f_j\right\}_{j\in\mathbb{Z}}$ be its Fourier coefficients. Then the sequence of $N\times N$ matrices $\left\{T_N\right\}_{N\in\mathbb{N}}$ with $T_N=[f_{i-j}]_{i,j=1}^{N}$ (resp. $\{ H_N \}_N$ with $H_N = [f_{i+j-2}]_{i,j=1}^N$) is called the \emph{sequence of Toeplitz (resp. Hankel) matrices generated by $f$}.
\end{definition}
\noindent
The generating function $f$ provides a description of the spectrum of $T_{N}$, for $N$ large enough in the sense of the following definition.

\newcommand{\p}{_{N\in\mathbb{N}}}

\begin{definition}\label{def-distribution}
	Let $f:[a,b]\to\mathbb{C}$ be a measurable function and let $\{A_N\}\p$ be a sequence of matrices of size $N$ with singular values $\sigma_j(A_N)$, $j=1,\ldots,N$. We say that $\{A_N\}\p$ is {\em distributed as $f$ over $[a,b]$ in the sense of the singular values,} and we write $\{A_N\}\p\sim_\sigma(f,[a,b]),$ if
	\begin{align}\label{distribution:sv-eig-bis}
		\lim_{N\to\infty}\frac{1}{N}\sum_{j=1}^{N}F(\sigma_j(A_N))=
		\frac1{b - a} \int_a^b F(|f(t)|) dt,
	\end{align}
	for every continuous function $F$ with compact support. In this case, we say that $f$ is the \emph{symbol} of $\{A_{N}\}_{N}$.
\end{definition}
\noindent
In the special case $f \equiv 0$, we say that $\{A_N\}_{N\in \mathbb{N}}$ is a zero distributed sequence. The above relation tells us that in presence of a zero distributed sequence 
the singular values of the $N$th matrix (weakly) cluster around $0$. This can be formalized
by the following result \cite{GS}.

\begin{proposition}
	Let $\{ A_N \}_N$ be a matrix sequence. Then
	$\{ A_N \}_N \sim_{\sigma} 0$ if and only if there
	exist
	two matrix sequences $\{ R_N \}_N$ and $\{ E_N\}_N$ such that
	$A_N = R_N + E_N$, and
	\[
	\lim_{N \to \infty} \frac{\mathrm{rank}(R_N)}{N} = 0, \qquad
	\lim_{N \to \infty} \norm{E_N} = 0.
	\]
\end{proposition}
\noindent
For our off-diagonal analysis
we need to characterize the symbol of Hankel matrices \cite{Fasino}.

\begin{proposition} \label{lem:hankel-zero-symbol}
	If $\{ H_N \}_N$ is an Hankel sequence generated
	by $f \in L^1$, then
	$\{ H_N \}_N \sim_\sigma 0$.
\end{proposition}

\begin{theorem}
	Let $\{ T_N \}_N$ be a sequence of Toeplitz matrices generated
	by $f \in L^1$. Then, for every off-diagonal block sequence
	$\{ Y_N \}_N$ of $\{ T_N \}_N$ with $Y_N\in\mathbb R^{\hat N\times \hat M}$, $\hat N,\hat M<N$
	there exist two sequences $\{ \hat R_N \}_N$ and $\{ \hat E_N \}_N$
	such that $Y_N = \hat R_N + \hat E_N$ and
	\[
	\lim_{N \to \infty} \frac{\mathrm{rank}(\hat R_N)}{N} = 0, \qquad
	\lim_{N \to \infty} \norm{\hat E_N} = 0.
	\]
\end{theorem}

\begin{proof}
	Consider the following partitioning of $T_N$
	\begin{equation*}
		T_{N}=
		\begin{bmatrix}
			T_{11} & T_{12}\\
			T_{21} & T_{22}
		\end{bmatrix},
	\end{equation*}
	where $T_{11}$ and $T_{22}$ are square. Without loss
	of generality, we assume that the off-diagonal block $Y_N$ is
	contained in $T_{21}$. Denote by $\{ H_N \}_N$ the Hankel sequence
	generated by $f$, the same
	function generating $\{ T_N \}_N$, and let $J$ be the counter-identity, with ones
	on the anti-diagonal and zero elsewhere.
	Then, $T_{21} J$ is a submatrix of $H_N$. Notice that $H_N$ does
	not depend on the specific choice of partitioning. In view
	of Proposition~\ref{lem:hankel-zero-symbol} we can
	write $H_N = R_N + E_N$, and therefore
	$T_{21}$ is a submatrix of $R_N J + E_N J$.
	We denote by $\hat R_{N}$ and $\hat E_{N}$ these two submatrices;
	since $\mathrm{rank}(\hat R_{N}) \leq
	\mathrm{rank}(R_{N})$ and $\norm{\hat E_{N}} \leq \norm{E_N}$, we have
	\[
	\lim_{N \to \infty} \frac{\mathrm{rank}(\hat R_{N})}{N} = 0, \qquad
	\lim_{N \to \infty} \norm{\hat E_{N}} = 0.
	\]
	$Y_N$ is a subblock of either $T_{21}$ or $T_{12}$,
	so the claim follows.
\end{proof}

The above result has an immediate consequence concerning the $\epsilon$-qsrank of a sequence of Toeplitz
matrices $\{ T_N \}_N$, and of $\{ T_N + Z_N \}_N$, where
$Z_N$ is any zero distributed matrix sequence.

\begin{corollary} \label{cor:quasisep-glt}
	Let $\{ T_N + Z_N \}_N$ be a sequence of matrices with $T_N$ Toeplitz generated
	by $f \in L^1$, and $Z_N$ zero distributed.
	Then, there exists a sequence
	of positive numbers
	$\epsilon_N \to 0$, such that
	\[
	\lim_{N \to \infty} \frac{\mathrm{qsrank}_{\epsilon_N}(T_N + Z_N)}{N} = 0.
	\]
\end{corollary}
\noindent
Corollary~\ref{cor:quasisep-glt} guarantees that the $\epsilon$-qsrank will grow slower than $N$ for an infinitesimal choice of truncation $\epsilon_N$.

In the finite differences case, the Toeplitz matrix $T_{\alpha, N}$ for the discretization
of the Gr\"unwald-Letnikov formulas is generated by a function $f$, which
is in $L^1$ as a consequence of the decaying property of fractional binomial coefficients. Therefore,
we expect the matrix $\fdm$ in \eqref{system_FD}, defined
by diagonal scaling of $T_{\alpha, N}$ and its transpose, to have off-diagonal blocks with low numerical rank. 	

In case of high-order finite elements with maximum regularity defined on uniform meshes, a technique similar to the one used in \cite[Chapter 10.6]{GS} can be employed to prove that the sequence of the coefficient matrices is a low-rank perturbation of a diagonal-times-Toeplitz sequence --- a structure that falls again under the GLT theory. Corollary~\ref{cor:quasisep-glt} can then be applied to obtain that the quasiseparable rank of these discretizations grows slower than $N$.

\subsection{Finite differences discretization}\label{sec:anal1}
Matrices stemming from finite difference discretizations have
the form $\fds =  -\Delta x^{-\alpha}(D_+^{(m)} T_{\alpha,N} + D_-^{(m)} T_{\alpha,N}^T)$ for the steady state scenario or $\fdm =  I - \Delta t \fds$ in the time dependent case, see \eqref{system_FD}. In order to bound the $\epsilon$-qsrank of $\fds,\fdm$ we need  to look at the off-diagonal blocks of  $T_{\alpha,N} $.
 To this aim, we exploit some recent results on the singular values decay of structured matrices.

Let us begin by recalling a known fact about Cauchy matrices.

\begin{lemma}[Theorem A in \cite{fiedler2010notes}]\label{lem:cauchy-pos-def}
	Let $\mathbf x, \mathbf y$ two real vectors of length $N$,
	with ascending and descending ordered entries, respectively.
	Moreover, we denote with  $C(\mathbf x, \mathbf y)$ the Cauchy matrix defined by
	\[
	C_{ij} = \frac{1}{x_i - y_j}, \qquad
	i,j = 1, \ldots, N.
	\]
	If $C(\mathbf x, \mathbf y)$ is symmetric and
	$x_i \in [a, b]$ and $y_j \in [c, d]$ with $a > d$,
	then $C(\mathbf x, \mathbf y)$ is positive definite.
\end{lemma}
We combine the previous result with a technique inspired by \cite{bhatia2006infinitely}, to prove that a Hankel matrix built with the binomial coefficients arising in the Gr\"unwald-Letnikov expansion is positive semidefinite.
\begin{lemma} \label{lem:hadamard-product}
	Consider the Hankel matrix $H$ defined as
	\[
	H = (h_{ij}), \qquad
	h_{ij} = g_{i+j}^{(\alpha)},
	\]
	for $1 \leq \alpha \leq 2$. Then, $H$ is positive semidefinite.
\end{lemma}

\begin{proof}
	Observe that for $k\geq 2$ we can rewrite $g_{k}^{(\alpha)}$ as follows:
	\begin{align*}
	g_{k}^{(\alpha)} &= \frac{(-1)^k}{k!} \alpha (\alpha - 1) \ldots (\alpha - k + 1) \\
	&= \frac{\alpha (\alpha - 1)}{k!} (k-\alpha-1) (k-\alpha-2) \ldots (2 - \alpha) \\
	&= \alpha (\alpha - 1) \frac{\Gamma(k-\alpha)}{\Gamma(k+1) \Gamma(2-\alpha)}.
	\end{align*}
	By using the Gauss formula for the gamma function:
	\[
	\Gamma(z) = \lim_{m \to \infty} \frac{m! m^z}{z (z + 1) (z + 2) \ldots (z+m)}, \quad
	z \neq \{ 0, -1, -2, \ldots \},
	\]
	we can rewrite the entries of the matrix $H$ as
	\[
	g_{k}^{(\alpha)} = \alpha (\alpha - 1) \lim_{m \to \infty} \frac{1}{m! m^3} \prod_{p = 0}^{m} \frac{k+1+p}{k-\alpha+p} (2 - \alpha + p).
	\]
	This implies that the matrix $H$ can be seen as the limit of
	Hadamard products of Hankel matrices. Since positive semidefiniteness is preserved
	by the Hadamard product (Schur product theorem) and by the limit operation \cite{bhatia2006infinitely}, if the Hadamard
	products
	\[
	H_0 \circ \ldots \circ H_{m}, \qquad (H_p)_{ij} = \frac{i+j+1+p}{i+j-\alpha+p}
	\]
	are positive semidefinite for every $m$ then $H$ is also
	positive semidefinite. Notice that we can write
	\[
	(H_p)_{ij} = \frac{i+j+1+p}{i+j-\alpha+p} = 1 + \frac{\alpha+1}{i+j-\alpha+p}
	\]
	that can be rephrased in matrix form as follows:
	\[
	H_p = \mathbf e \mathbf e^T +
	(\alpha + 1) \cdot C(\mathbf x, - \mathbf x), \qquad
	\mathbf x =
	\begin{bmatrix}
	1 \\
	\vdots \\
	N
	\end{bmatrix}  + \frac{p-\alpha}{2} \mathbf e, \qquad
	\mathbf e =
	\begin{bmatrix}
	1 \\
	\vdots \\
	1
	\end{bmatrix}.
	\]
	
	All the components of $\mathbf x$ are positive, since $\alpha < 2$.
	This implies, thanks to Lemma~\ref{lem:cauchy-pos-def}, that the Cauchy
	matrix $C(\mathbf x, -\mathbf x)$ is positive definite. Summing it
	with the positive semidefinite matrix on the left retains this
	property, so $H_p$ is positive semidefinite as well.
\end{proof}

The next result ensures that positive semidefinite
Hankel matrices are numerically low-rank.

\begin{lemma}[Theorem 5.5 in \cite{Beckermann2016}] \label{beckermann}
	Let $H$ be a positive semidefinite Hankel matrix of size $N$. Then,
	the $\epsilon$-rank of $H$ is bounded by
	\[
	\mathrm{rank}_{\epsilon}(H) \leq 2 + 2\left\lceil
	\frac{2}{\pi^2} \log \left( \frac {4}{\pi}N \right)
	\log\left(\frac{16}{\epsilon}\right)
	\right\rceil =: \mathfrak B(N, \epsilon).
	\]
\end{lemma}
We are now ready to state a bound for the $\epsilon$-qsrank of $T_{\alpha,N}$.
\begin{lemma} \label{lem:fd-qsrank}
	Let $T_{\alpha, N}$ be the lower Hessenberg Toeplitz matrix defined in \eqref{toeplitz_FD}.
	Then, for every $\epsilon > 0$, the $\epsilon$-qsrank of $T_{\alpha, N}$
	is bounded by
	\[
	\mathrm{qsrank}_{\epsilon} (T_{\alpha,N})
	\leq \mathfrak B\left(N, \frac{\epsilon}{2}\right) =
	2 + 2\left\lceil
	\frac {2}{\pi^2} \log \left( \frac{4}{\pi}N \right)
	\log\left(\frac{32}{\epsilon}\right)
	\right\rceil.
	\]
\end{lemma}

\begin{proof}
	We can verify the claim on the lower triangular part, since every off-diagonal block in the upper one has rank at most $1$.
	Let $Y \in \mathbb C^{s \times t}$ 
	be any lower off-diagonal block of $T_{\alpha,N}$. Without loss of generality we can assume that $Y$ is maximal, i.e. $s+t=N$. In fact, if $\mathrm{rank}(Y+\delta Y)=k$ and $\norm{\delta Y}_2\leq \epsilon\norm{T_{\alpha, N}}_2$ then the submatrices of $\delta Y$ verify the analogous claim for the corresponding submatrices of $Y$. 
	
	The entries of $Y$ are given by $Y_{ij} = -g_{1+i-j+t}^{(\alpha)}$. 
	Let $h := \max\{s, t\}$, and the $h \times h$ matrix
	$A$ defined by $A_{ij} := -g_{1+i-j+h}^{(\alpha)}$. It is immediate to verify
	that $Y$ coincides with  either the last $t$ columns or the first $s$ rows 
	of $A$. In fact, for every $1 \leq i \leq s$ and 
	$1 \leq j \leq t$ we have $Y_{ij} = -g_{1+i-j+t}^{(\alpha)} = -g_{1+i-(j-t+h)+h}^{(\alpha)} = A_{i,j-t+h}$. In particular,
	$Y$ is a submatrix of $A$ and therefore $\norm{Y}_2 \leq 
	\norm{A}_2$. Two possible arrangements of $Y$ and $A$ are pictorially described by the following figures. 
	\[
	\begin{minipage}{.45\linewidth}
	\centering
	\begin{tikzpicture}[scale=0.7]
	\node at (-2.9,1.5) { $A = $} ;
	\draw (0,0) rectangle (4,4);
	\draw (0,4) -- (4,0);
	\draw (0,3) -- (1,3) -- (1,0);
	\draw[dashed] (-2,3) rectangle (0,0);
	\node at (.5,1.5) {  $Y$ };
	\node at (-1,1) {  };
	\end{tikzpicture}
	\end{minipage}~\begin{minipage}{.45\linewidth}
	\centering
	\begin{tikzpicture}[scale=0.7]
	\node at (-0.9,-.5) { $A = $} ;
	\draw (0,0) rectangle (4,4);
	\draw (0,4) -- (4,0);
	\draw (0,1) -- (3,1) -- (3,0);
	\draw[dashed] (0,1) rectangle (3,-2);
	\node at (1.5,.5) {  $Y$ };
	\node at (-1,1) {  };
	\end{tikzpicture}
	\end{minipage}
	\]
	In order to estimate $\norm{A}_2$, we
	perform the following $2 \times 2$ block
	partitioning: 
	\[
	  A=\begin{bmatrix}A^{(11)}&A^{(12)}\\ A^{(21)}& A^{(22)}\end{bmatrix}, \qquad A^{(ij)}\in\mathbb C^{m_{ij}\times n_{ij}},\qquad 
	  \begin{cases}
	   m_{1j}=n_{i1}=\lceil\frac h2\rceil \\
	   m_{2j}=n_{i2}=\lfloor\frac h2\rfloor
	  \end{cases}. 
	\]
	Recalling that $h$ is the maximum dimension of the block
	$Y$, and therefore $h \leq N - 1$, 
	this choice yields $m_{ij}+n_{ij}\leq N$. In passing, 
	we remark that this partitioning is not necessarily conformal
	to $Y$, but is performed with the sole purpose of 
	estimating $\norm{A}_2$ by a constant times the norm
	of $T_{\alpha,N}$. Indeed, 
	we now consider the subdiagonal block $T^{(ij)}$ of 
	$T_{\alpha,N}$ defined by (using MATLAB-style notation)
	\[
	  T^{(ij)} := T_{\alpha, N}(N-m_{ij}+1:N, N-m_{ij}-n_{ij}+1:N-m_{ij}), \qquad 
	  i,j = 1,2
	\] which is of dimension $m_{ij}\times n_{ij}$
	and well defined because $m_{ij} + n_{ij} \leq N$.
	These blocks verify
	$|T^{(ij)}| \geq |A^{(ij)}|$ for 
	every $i,j = 1,2$, which can be verified using the 
	property $g_j^{(\alpha)}>g_{j+1}^{(\alpha)}>0$ for
	all $j\geq 2$
	(see Section~\ref{sec:grunw}).
	Since both $T^{(ij)}$ and $A^{(ij)}$ 
	are nonpositive, we have for the monotonicity of the
	$2$ norm that $\norm{A^{(ij)}}_2 \leq \norm{T^{(ij)}}_2$. 
	In addition, we exploit the relation
	\begin{align*}
	\norm{A}_2 &\leq \left\lVert\begin{bmatrix}A^{(11)}&\\ & A^{(22)}\end{bmatrix}\right\rVert_2 +\left\lVert\begin{bmatrix}&A^{(12)}\\ A^{(21)}& \end{bmatrix}\right\rVert_2 \\
	&= \max\{\norm{A^{(11)}}_2,\norm{A^{(22)}}_2\} +\max\{\norm{A^{(12)}}_2,\norm{A^{(21)}}_2\}
	\end{align*}	
	to get $\norm{A}_2\leq 2\norm{T_{\alpha,N}}_2$.
	
	Let $J$ be the $h \times h$ counter-identity; in light of  Lemma~\ref{lem:hadamard-product}, the matrix $-AJ$ is Hankel and positive semidefinite.  Applying  Lemma~\ref{beckermann} to $-AJ$ with truncation $\frac{\epsilon}{2}$ we obtain
	$
	\mathrm{rank}_{\frac \epsilon 2}(A) = \mathrm{rank}_{\frac \epsilon 2}(AJ) \leq \mathfrak B(N, \frac \epsilon 2).
	$
	Since $Y$ is a submatrix of $A$ there exists $\delta Y$ such that $\norm{\delta Y}_2\leq\epsilon\norm{T_{\alpha, N}}_2$ and $\mathrm{rank}(Y+\delta Y)\leq \mathfrak B(N, \frac \epsilon 2)$. So, we conclude
	that $\mathrm{qsrank}_{\epsilon}(T_{\alpha,N}) \leq \mathfrak B(N, \frac \epsilon 2)$.
\end{proof}

\begin{corollary} \label{cor:fd-qsrank}
	Let $\fdm = I + \frac{\Delta t}{\Delta x^\alpha} (D_+^{(m)} T_{\alpha,N} + D_-^{(m)} T_{\alpha,N}^T)$  be defined as in \eqref{system_FD}, where $D_+^{(m)}$ and
	$D_-^{(m)}$ contain the samplings of $d_+(x,t_m)$ and $d_-(x,t_m)$.
	Then:
	\[
	\mathrm{qsrank}_{\epsilon}(\fdm) \leq
	3 + 2\left\lceil
	\frac 2{\pi^2} \log \left( \frac{4}{\pi}N \right)
	\log\left(\frac{32}{\hat\epsilon}\right)
	\right\rceil, \quad
	\hat\epsilon := \frac{\norm{\fdm}}{\norm{T_{\alpha,N}} \cdot \max\{\norm{D_+^{(m)}}, \norm{D_-^{(m)}}\}} \epsilon.
	\]
\end{corollary}

\begin{proof}
	Clearly, the result is invariant under scaling, so we assume that 
	$\frac{\Delta t}{\Delta x^\alpha} = 1$. 
	Consider a generic off-diagonal block $Y$ of $\fdm$,
	and assume without loss of generality that it is in the lower
	triangular part.  If $Y$ does not intersect the first
	subdiagonal, then $Y$ 
	is a subblock of $D_+^{(m)} T_{\alpha, N}$, and
	so we know by Lemma~\ref{lem:fd-qsrank} that there
	exists a perturbation
	$\delta Y$
	with norm bounded by $\norm{\delta Y} \leq \norm{D_+^{(m)}} \norm{T_{\alpha,N}} \cdot \hat\epsilon$ such that $Y + \delta Y$ has rank at most $\mathfrak B(N, \frac{\hat\epsilon}{2})$.
	In particular, $\delta Y$ satisfies $\norm{\delta Y} \leq \norm{\fdm} \cdot \epsilon$.
	
	Since we have excluded one subdiagonal, for a generic off-diagonal
	block $Y$ we can find a perturbation with norm bounded
	by $\norm{\fdm} \cdot \epsilon$ such that $Y + \delta Y$
	has rank $1 + \mathfrak{B}(N, \frac{\hat\epsilon}{2})$.
\end{proof}

	\subsection{Finite element discretization}\label{sec:anal2}
We consider the left and
	right-handed fractional derivatives of order $1<\alpha<2$ in Riemann-Liouville form, defined
	as follows:
	\begin{align*} \frac{\partial^\alpha u(x,t)}{\partial^{RL}_{+} x^\alpha}&=
	\frac{1}{\Gamma(2-\alpha)}\frac{\partial^2}{\partial
		x^2}\int_L^x\frac{u(\xi,t)}{(x-\xi)^{\alpha-1}}d\xi,\\
	\frac{\partial^\alpha u(x,t)}{\partial^{RL}_{-}
		x^\alpha}&=\frac{1}{\Gamma(2-\alpha)}\frac{\partial^2}{\partial
		x^2}\int_x^R\frac{u(\xi,t)}{(\xi-x)^{\alpha-1}}d\xi,
	\end{align*} where
	$\Gamma(\cdot)$ is the gamma function. From now on we focus on the discretization of the left-handed derivative; the results for the right-handed one are completely analogous. In this section we consider
	the case of constant diffusion coefficients, but as outlined
	in Remark~\ref{rem:nonconstantidiffusion} this is not restrictive.
	
	Let us discretize the operator $\frac{\partial ^\alpha}{\partial^{RL}_+}$ by using a finite
	element method. More precisely, we choose a set of basis functions
	$\mathcal B:=\{\varphi_1,\dots,\varphi_N\}$, normalized
	so that $\int_L^R\varphi_i(x)\ dx = 1$. This leads to the stiffness  matrix
    $\femsp$ defined by
    \begin{equation}\label{eq:fe-int}
    (\femsp)_{ij}= \langle \varphi_i, \frac{\partial^\alpha \varphi_j}{\partial^{RL}_+ x^\alpha}\rangle=\frac{1}{\Gamma(2-\alpha)}\int_L^R
    \varphi_{i}(x)\frac{\partial^2}{\partial
    	x^2}\int_{L}^x\varphi_j(y)(x-y)^{-\alpha+1}dy\ dx.
    \end{equation}	

    	A key ingredient in our analysis is requiring  a separation property for the elements of the basis. This is formalized in the following definition.
    	
    	\begin{definition}\label{def:overlap} We say that the basis $\mathcal B$ has
    		\emph{$\delta$-overlapping $k$}, with $\delta\geq 0$, if $\forall j_1,j_2\in\{1,\dots, N\}$ such that $j_2-j_1\geq k$, there exists
    		$x_0\in[L,R]$ such that
    		\[ \supp(\varphi_j)\subset [L,x_0-\delta]\ ;\  j< j_1, \qquad
    		\supp(\varphi_j)\subset [x_0+\delta,R]\ ;\ j> j_2.
    		\] When $\delta=0$ we simply say that $\mathcal B$ has overlapping
    		$k$.
    	\end{definition}
    	The property of being a basis with $\delta$-overlapping $k$ is described
    	pictorially in Figure~\ref{fig:overlapping}.

    Our strategy for proving the presence of the (approximate) quasiseparable structure in $\femsp$, is to show that any off-diagonal block can be approximated summing a few integrals of separable functions. 	
    In view of the following result, this implies the low-rank structure.
    	\begin{lemma} \label{lem:separable-integral}Let $(g_i)_{i=1}^m, (h_j)_{j=1}^n$ be families of
    		functions on
    		$[
    		a,b
    		]$
    		and
    		$[
    		c,d
    		]$ and define $\Gamma_{ij}(x,y)=g_i(x)h_j(y)$ for $i = 1, \ldots, m$ and $j = 1, \ldots, n$.
    		Consider the functional
    		\[
    		I(\Gamma_{ij}) = \int_{a}^{b} \int_{c}^{d} \Gamma_{ij}(x,y) \ dx \ dy
    		\]
    		Then, the matrix $X =(x_{ij})$ with $ x_{ij}:=I(\Gamma_{ij})\in\mathbb{C}^{m\times n}$ has rank $1$.
    	\end{lemma}
    	
    	\begin{proof}
    		For a fixed $i,j$, we can write
    		\[
    		x_{ij} = I(\Gamma_{ij}) = \int_{a}^{b} \int_{c}^{d} \Gamma_{ij}(x,y)\ dx \ dy
    		= \int_{a}^{b} g_i(x) \ dx \int_{c}^{d} h_j(y) \ dy.
    		\]
    		Then $X=GH^T$ where the column vectors $G,H$ have entries $G_i=\int_{a}^{b} g_i(x) \ dx$ for $i=1,\dots,m$ and $H_j=\int_{c}^{d} h_j(y) \ dy$ for $j=1,\dots,n$, respectively.
    	\end{proof}

	\begin{figure}
		\centering
		\begin{tikzpicture}
		\draw[<->] (-4.3,0) -- (4.3,0);
		\draw (-4,.1) -- (-4,-.1) node[anchor=north] {$-1$};
		\draw (4,.1) -- (4,-.1) node[anchor=north] {$1$};
		\draw (0,.1) -- (0,-.1) node[anchor=north] {$0$};
		
		\foreach \x in { -4, -3.65, ..., -.5, .55, .9, ..., 3 }
		{
			\draw[line width=0.5pt, gray] plot[smooth]
			coordinates{ (\x,0) (\x+.1,.1) (\x+.5,1)  (\x+.9,.1) (\x+1,0) };
		}
		
		\foreach \x in { -.5, -.15, 0.2 }
		{
			\draw[blue, line width = 1pt] plot[smooth] coordinates{ (\x,0) (\x+.1,.1) (\x+.5,1)  (\x+.9,.1) (\x+1,0) };
		}
		
		\draw[dashed] (.35,.1) -- (.35,-.7) node[anchor=north] {$x_0$};
		\end{tikzpicture}
		\caption{A pictorial representation of a basis with $\delta$-overlapping $2$ for any small
			enough $\delta$.}
		\label{fig:overlapping}
	\end{figure}
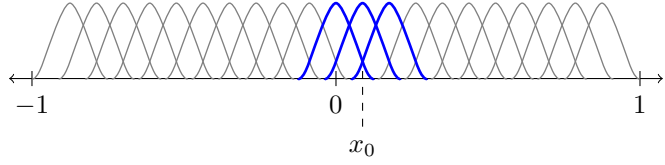
	
The only non separable part of \eqref{eq:fe-int} is
the function
	$g(x,y) = (x - y)^{- \alpha}$. A separable  approximation of  $g(x,y)$ on $[a,b]\times [c,d]$ with $a>d$ can be obtained by providing an approximation for $(x'+y')^{-\alpha}$ on $[a',b']^2$, where $a'=\frac{a-d}2$ and $b'=\max\{b-a,d-c\}+a'$, by means of the change of variables $x'=x-(a+d)/2$ and $y'=(a+d)/2- y$. Therefore, we
state the following result, whose proof follows the line of a
	similar statement for  $\alpha = 1$ in
	\cite{grasedyck2001singular}. Analogous estimates, for other kind of kernel functions can be found in \cite[Chapter 4]{hackbusch2015hierarchical}.
	
	\begin{lemma} \label{lem:alpha-separable}
		Let $g(x,y) = (x + y)^{-\alpha}$, with $\alpha > 0$,
		and consider
		the square $I^2$, with $I = [a, b]$ and $a > 0$. Then,
		for any $\epsilon > 0$, there exists a
		function $g_{\epsilon}(x, y)$ satisfying
		\begin{equation} \label{eq:gepsilon}
		| g(x,y) - g_\epsilon(x, y) | \leq | g(x,y) | \epsilon,
		\qquad
		x,y \in I,
		\end{equation}
		and $g_\epsilon(x,y)$ is the sum of at most $k_{\epsilon}$ separable functions where
		\begin{equation} \label{eq:kepsilon}
		k_{\epsilon} =  2\left\lceil\log_2 \left(\frac ba \right)\right\rceil \cdot
		\left(1 + \left\lceil \log_2\left(\frac{\alpha \cdot 4^\alpha}{\epsilon}\right) \right\rceil \right).
		\end{equation}
	\end{lemma}
	
	\begin{proof}
		Consider the partitioning of the interval $I$ given by
		$I = I_0 \cup \ldots \cup I_K$ where
		\[
		I_j = \begin{cases}
		[ a + 2^{-j-1} \Delta, a + 2^{-j} \Delta ] & 0 \leq j < K \\
		[ a, a + 2^{-K} \Delta ] & j = K \\
		\end{cases}, \qquad
		\Delta := b - a, \qquad
		K = \left\lceil \log_2 \left(\frac ba\right) \right\rceil,
		\]
		and denote by $c_j$ the midpoint of $I_j$.
		The partitioning can be pictorially described as follows
		for $K = 4$.
		\[
		\begin{tikzpicture}
		\draw (0,0) -- (10,0);
		\draw (1,-0.1) -- (1,0.1) node[anchor=south] {$a$};
		\draw (9,-0.1) -- (9,0.1) node[anchor=south] {$b$};		
		\foreach \j in {0, ..., 3} {
			\pgfmathparse{pow(2.0, -\j)};
			\xdef\mydelta{\pgfmathresult}
			\draw (1+4 * \mydelta, -0.1) -- (1+4 * \mydelta, 0.1);
			\node at (1 + 4 * \mydelta * 1.5, -.3) {$I_\j$};
		}
		\node at (1.25, -.3) {$I_4$};
		\end{tikzpicture}
		\]
		Notice that the choice of $K$ implies that $\mathrm{diam}(I_K) \leq a$. In general, the left endpoint of each sub-interval is greater than the diameter, i.e., $\mathrm{diam}(I_j)\leq a + 2^{-j-1}\Delta$. This provides the inequality $c_j\geq \frac 32\mathrm{diam}(I_j)$  for every $I_j$.
		
		Starting from the subdivision of $I$ we get the following partitioning of $I^2$:
		\[
		\begin{tikzpicture}[scale=1.2]
		\draw (0,0) rectangle (4,4);
		\draw (0,2) -- (4,2); \node at (3,3) {$I_0 \times I_0$};
		\draw (2,0) -- (2,4); \node at (3,1) {$I_0 \times \hat I_0$};
		\draw (1,0) -- (1,2); \node at (1,3) {$\hat I_0 \times I_0$};
		\draw (0,1) -- (2,1); \node at (1.5,1.5) {$I_1 \times I_1$};
		\draw (.5,0) -- (.5,1); \node at (1.5,.5) {$I_1 \times \hat I_1$};
		\draw (0,.5) -- (1,.5); \node at (.5,1.5) {$\hat I_1 \times I_1$};
		\end{tikzpicture},
		\]
		where $\hat I_i = \bigcup_{j = i+1}^K I_j$.
		For any of the domains, we can consider the Taylor expansions
		of $g(x,y)$ either in the variable $x$ or in the variable $y$, expanded at
		the point $c_i$.
		\[
		g_{x,i}(x,y) := \sum_{j \geq 0} \frac{1}{j!} \frac{\partial^j}{\partial x^j} g(c_i,y) (x - c_i)^j, \qquad
		g_{y,i}(x,y) :=
		\sum_{j \geq 0} \frac{1}{j!}
		\frac{\partial^j}{\partial y^j} g(x, c_i) (y - c_i)^j.
		\]
		Using the fact that $\frac{\partial^j}{\partial x^j} g(c_i, y) =
		\frac{\Gamma(\alpha+j)}{\Gamma(\alpha)}(c_i + y)^{-\alpha-j}$
		(and similarly in the $y$ variable) we can rephrase
		the above expansion as follows:
		\[
		g_{x,i}(x,y) = \sum_{j \geq 0} \frac{\Gamma(\alpha+j)}{\Gamma(j+1) \Gamma(\alpha)}
		\left(\frac{x - c_i}{c_i + y}\right)^j
		\underbrace{(c_i + y)^{-\alpha}}_{g(c_i, y)}
		\]
		and similarly for $g_{y,i}(x,y)$.
		We now approximate $g(x,y)$ using truncations of the above
		expansions on each of the sets in the partitioning of the
		square. Consider the sets of the form $I_i \times I_i$
		or $I_i \times \hat I_i$. We define
		\[
		g_{N,x,i}(x, y) = \sum_{j = 0}^N \frac{\Gamma(\alpha+j)}{\Gamma(j+1) \Gamma(\alpha)}
		\left(\frac{x - c_i}{c_i + y}\right)^j (c_i + y)^{-\alpha}.
		\]
		
		Observe that since $c_i$ is the midpoint of $I_i$,
		$|x - c_i| \leq \frac 12 \mathrm{diam}(I_i)$.
		In addition, since both $c_i$ and $y$ are positive we
		have
		\[
		|c_i + y| = c_i + y \geq c_i 
		\geq \frac 32 \mathrm{diam}(I_i).
		\]
		Therefore, we have
		$|x - c_i| \leq \frac 13 (c_i + y)$, so we can bound
		\[
		|g(x,y) - g_{N,x,i}(x,y)| \leq \frac{|g(c_i,y)|}{\Gamma(\alpha)} \sum_{j \geq N + 1} \frac{\Gamma(\alpha+j)}{\Gamma(j+1)}
		\left(\frac 13 \right)^j .
		\]
		One can easily check that $\frac{\Gamma(\alpha+j)}{\Gamma(j+1)}
		\leq \frac{(\lfloor \alpha \rfloor + j)!}{j!}$, hence
		 we can write
		\begin{align*}
		\sum_{j = N+1}^\infty \frac{\Gamma(\alpha+j)}{\Gamma(j+1)}
		\left( \frac 13\right)^j  &\leq
		\sum_{j = N+1}^\infty \frac{(\lfloor \alpha \rfloor + j)!}{j!}
		\frac{1}{3^j} =
		\left( \sum_{j = N+1}^\infty  \frac{(\lfloor \alpha \rfloor + j)!}{j!}
		x^j \right)\Bigg|_{x = \frac 13} \\
		&= \left( \frac{d^{\lfloor \alpha \rfloor}}{dx^{\lfloor \alpha \rfloor}} \sum_{j = N+1}^\infty
		x^{\lfloor \alpha \rfloor+j} \right)\Bigg|_{x = \frac 13}
		= \left( \frac{d^{\lfloor \alpha \rfloor}}{dx^{\lfloor \alpha \rfloor}}
		\frac{x^{N+1+\lfloor \alpha \rfloor}}{1 - x} \right)\Bigg|_{x = \frac 13}.
		\end{align*}
		
		The last quantity can be bounded using the relation
		$|f^{(k)}(w)| \leq k! \max_{|z - w| = r}  |f(z)| \cdot r^{-k}$, with $r$ being a positive number such that
		$f(w)$ is analytic for $|z - w| \le r$. Note that $|f(z)|$
		assumes the maximum at the rightmost point of the circle $|z - w| = r$, since
		there we have both the maximum of the numerator and the minimum of
		the denominator. Choosing
		$r = \frac 16$ provides
		\[
		\left( \frac{d^{\lfloor \alpha \rfloor}}{dx^{\lfloor \alpha \rfloor}}
		\frac{x^{N+1+\lfloor \alpha \rfloor}}{1 - x} \right)\Bigg|_{x = \frac 13} \leq 3^{\lfloor \alpha \rfloor} \cdot \lfloor \alpha \rfloor!
		\cdot 2^{-N}.
		\]
		Plugging this back into our bound yields
		\[
		|g(x,y) - g_{N,x,i}(x,y)| \leq \frac{3^{\lfloor \alpha \rfloor} \cdot \lfloor \alpha \rfloor!
			\cdot 2^{-N}}{\Gamma(\alpha)} \cdot
		|g(c_i, y)|.
		\]
		Moreover, using $|x - c_i| \leq \frac 13 (c_i + y)$, we have that \begin{align*}
		g(x,y)=(x-c_i + c_i+y)^{-\alpha}\ge \left(\frac{c_i+y}3+c_i+y\right)^{-\alpha}= \left(\frac 34\right)^\alpha(c_i+y)^{-\alpha}= \left(\frac 34\right)^\alpha g(c_i,y)	
		\end{align*}
		for any $x \in I_i$. Therefore,
		\begin{equation} \label{eq:separable-alpha}
		|g(x,y) - g_{N,x,i}(x,y)| \leq
		\left(\frac 43\right)^\alpha\frac{3^{\lfloor \alpha \rfloor} \cdot \lfloor \alpha \rfloor!
			\cdot 2^{-N}}{\Gamma(\alpha)} \cdot
		|g(x, y)|, \quad
		(x,y) \in (I_i \times I_i) \cup (I_i \times \hat I_i).
		\end{equation}
		We can obtain an analogous result
		for the sets of the form $\hat I_i \times I_i$
		by considering the expansion $g_{N,y,i}(x,y)$.
		We define an approximant to $g(x,y)$ on $I^2$ by
		combining all the ones on the partitioning:
		\[
		g_N(x,y) := \begin{cases}
		g_{N,x,i}(x,y) & \text{on } I_i \times I_i \text{ and }
		I_i \times \hat I_i\\
		g_{N,y,j}(x,y) & \text{on } \hat I_i \times I_i
		\end{cases}.
		\]
	
		The function $g_N(x, y)$ is obtained summing $2K + 1$
                separable functions of order $N + 1$, which in turn implies that
                $g_N(x,y)$ can be written as a separable function of order $(2K + 1) \cdot (N + 1)$.
                Using that $2K+1 \leq 2 \lceil\log_2(\frac ba)\rceil$,
                we have that $g_N(x,y)$ is a $2\lceil \log_2(\frac ba) \rceil \cdot (N+1)$-separable
                approximant
		of $g(x,y)$ on $I^2$. We determine $N$ such that
		$
		\left(\frac 43\right)^\alpha\frac{3^{\lfloor \alpha \rfloor} \cdot \lfloor \alpha \rfloor!
			\cdot 2^{-N}}{\Gamma(\alpha)} \leq \epsilon
		$.
		Noting that $3^{\lfloor \alpha \rfloor - \alpha} \leq 1$,
		$\lfloor \alpha \rfloor ! / \Gamma(\alpha) \leq \alpha$,
		and setting
		$k_{\epsilon} =  2\lceil \log_2\left(\frac ba\right) \rceil \cdot (N+1)$
		we retrieve \eqref{eq:gepsilon}.
	\end{proof}
	
	\begin{remark}
		We note that the result of the previous lemma is slightly suboptimal, since
		we have chosen the fixed value of $r = \frac 16$ --- whereas the optimal one
		would be
		$r = \arg \min_{\rho} \rho^{-\lfloor \alpha \rfloor} \cdot \max_{|z - \frac 13|=\rho}
		  |f(z)|$. The latter leads to bounding the tail of the Taylor expansion
		with a quantity decaying as $\mathcal O(N^{\lfloor \alpha \rfloor} 3^{-N})$.
		The advantage of our formulation is that allows to
		explicitly bound $k_{\epsilon}$ with a multiple of $\log_2(\epsilon^{-1})$.
	\end{remark}
Lemma~\ref{lem:alpha-separable} enables to study the rank of the off-diagonal blocks in $\femsp$. Here we consider the constant coefficient case;
the generalization to variable coefficients requires little changes
as outlined in Remark~\ref{rem:nonconstantidiffusion}.
	
	\begin{theorem} \label{thm:structure-riesz-fem}
		Let $\femsp\in\mathbb C^{N\times N}$ be the matrix defined in \eqref{eq:fe-int} with $d_+(x) \equiv 1$.
		Assume that $\mathcal
		B$ has $\delta$-overlapping $k$ with $\delta>0$, and that
		the basis functions $\varphi_i(x) \geq 0$ are normalized to have
		$\int_{L}^R \varphi_i(x)\ dx = 1$.  Then
		\[ \
		\mathrm{qsrank}_\epsilon(\femsp)\le k + k_\epsilon = k +
		2\left\lceil\log_2 \left(\frac{R-L}{\delta} \right)\right\rceil \cdot
		\left(1 + \left\lceil \log_2 \left(\frac{(\alpha+1) \cdot 4^{\alpha+1}}{\epsilon}\right) \right\rceil \right).
		\]
	\end{theorem}
	
	\begin{proof}
		Let $Y$ be any off-diagonal block of $\femsp$ which does not
		involve any entry of the central $2k+1$ diagonals. Without loss of
		generality we can assume $Y$ to be in the lower left corner of $\femsp$. In
		particular, there exist $h,\ell$ such that $\ell-h\ge k$ and
		\begin{equation} \label{eq:Yij}
		   Y_{ij}= \langle\varphi_{i+\ell},\frac{\partial^\alpha}{\partial^{RL}_+}\varphi_j\rangle,\qquad
		 i=1,\dots,N-\ell,\quad j=1,\dots h.
		\end{equation}
		Since we are considering a basis with $\delta$-overlapping $k$,
		we can identify $x_0$ such that the support of $\varphi_{i+\ell}$
		is always contained in $[x_0+\delta, R]$ and the one of
		$\varphi_j$ in $[L, x_0 - \delta]$. Therefore,
		expanding the scalar product we obtain
		\begin{align*} Y_{ij}&=\frac{1}{\Gamma(2-\alpha)}\int_L^R
		\varphi_{i+\ell}(x)\frac{\partial^2}{\partial
			x^2}\int_{L}^x \frac{\varphi_j(y)}{(x-y)^{\alpha-1}}dy\ dx\\ &=\frac{1}{\Gamma(2-\alpha)}
		\int_{x_0 + \delta}^{R}	\varphi_{i+\ell}(x)\frac{\partial^2}{\partial
			x^2}\int_{L}^{x_0 - \delta} \frac{\varphi_j(y)}{(x-y)^{\alpha-1}}dy\ dx\\
		&= \frac{1}{\Gamma(2-\alpha)}\int_{x_0 + \delta}^{R}
		\varphi_{i+\ell}(x)\int_{L}^{x_0 - \delta} \frac{\partial^2}{\partial
			x^2}\frac{\varphi_j(y)}{(x-y)^{\alpha-1}}dy\ dx\\
		&=\frac{1}{\Gamma(2-\alpha)}\int_{x_0+\delta}^R
		\varphi_{i+\ell}(x)\int_{L}^{x_0-\delta} \frac{\alpha(\alpha - 1)\varphi_j(y)}{(x-y)^{\alpha+1}}dy\
		dx\\
		&=\int_{x_0+\delta}^R\int_L^{x_0-\delta}\frac{\alpha (\alpha - 1)}{\Gamma(2-\alpha)}\frac{\varphi_{i+\ell}(x)\varphi_j(y)}{(x-y)^{\alpha+1}}dy\
		dx.
		\end{align*}
		By the  change of variable $\hat x = x - x_0$,
		and $\hat y = x_0 - y$, we can write
		\[
		\frac{1}{(x - y)^{\alpha+1}} = \frac{1}{(\hat x + \hat y)^{\alpha+1}}, \qquad
		\begin{cases}
		\hat x \in [\delta, R-x_0] \subseteq [\delta, R - L] \\
		\hat y \in [\delta, x_0-L] \subseteq [\delta, R - L]
		\end{cases}
		\]
		Applying Lemma~\ref{lem:alpha-separable} to the right-hand side,
		on the larger interval $[\delta, R - L]$
		we recover a separable approximation of $(\hat x + \hat y)^{-\alpha-1}$. Since the change of variable does not mix $x$ and $y$, this also gives
		the relatively accurate separable approximation:
		\[
		f_{ij}(x,y) := \frac{\alpha (\alpha - 1)}{\Gamma(2-\alpha)} \frac{\varphi_{i+\ell}(x)\varphi_j(y)}{(x-y)^{\alpha+1}} = s_{k_\epsilon}(x,y) + r_{k_\epsilon}(x,y),
		\]
		with $s_{k_\epsilon}(x,y)$ sum of $k_\epsilon$ separable functions,
		with $k_\epsilon$ as in \eqref{eq:kepsilon},
		and $|r_{k_\epsilon}(x,y)| \leq
		|f_{ij}(x,y)| \cdot \epsilon$. 	
		Therefore, we can decompose $Y$ as $Y = S + E$, with
		\[
		S_{ij} = \int_{x_0+\delta}^R\int_L^{x_0-\delta} s_{k_\epsilon}(x,y) dx \ dy
		\]
		and $E$ defined analogously using $r_{k_\epsilon}(x,y)$ in
		place of $s_{k_\epsilon}(x,y)$. Lemma~\ref{lem:separable-integral} tells us that the rank
		of $S$ is bounded by $k_\epsilon$. On the other hand,
		using that $f_{ij}(x,y) \geq 0$, we obtain
		\[
		|E_{ij}| =  \left|\int_{x_0+\delta}^R\int_L^{x_0-\delta} r_{k_\epsilon}(x,y) dx \ dy \right| \leq
		\left|\int_{x_0+\delta}^R\int_L^{x_0-\delta} f_{ij}(x,y)\ dx \ dy \right| \cdot \epsilon \leq
		|Y_{ij}| \cdot \epsilon.
		\]

		This implies that $\norm{E} \leq \norm{Y} \cdot \epsilon$,
		so
		$\mathrm{rank}_{\epsilon}(Y) \leq k_\epsilon$. Since we have
		excluded the central $2k+1$ diagonals from our analysis,
		and $\norm{Y} \leq \norm{\femsp}$, we conclude that  $\mathrm{qsrank}_{\epsilon}(\femsp) \leq k_\epsilon + k$.
	\end{proof}

	\begin{remark} \label{rem:nonconstantidiffusion}
		We notice that the proof of Theorem~\ref{thm:structure-riesz-fem} remains unchanged if in place of \eqref{eq:Yij} one considers
		$\langle \varphi_i, d_+(x,t) \frac{\partial^\alpha \varphi_j}{\partial^{RL}_+ x^\alpha} \rangle$, with
		a positive diffusion coefficient $d_+(x,t)$. This means
		that the rank structure in $\femsp$ is present also in the
		non-constant diffusion case.
		The analogous statement is true for $\femsm$.
	\end{remark}

	\subsection{Approximating the 1D fractional discretizations in practice}
	\label{sec:practice}
	
	We have shown in the previous sections that several discretizations
	of fractional derivative operators are well-approximated by
	quasiseparable matrices of low-rank. However, we have not
	yet shown how to efficiently represent and compute such matrices.
		In fact, to make large scale computations feasible, we need
	to reduce both the storage and computational complexity to
	at most linear polylogarithmic cost. To this aim, we introduce
	HODLR matrices.
	
	\subsubsection{Hierachically off-diagonal low-rank matrices}
	Off-diagonal rank structures are often present in the
	discretization of PDEs and integral equations; this can be exploited using Hierarchical matrices ($\mathcal H$-matrices)
	\cite{hackbusch1999sparse,borm,hackbusch2015hierarchical} and their variants HSS, $\mathcal
	H^2$-matrices. The choice of the representation is usually tailored to the rank structure of the operator.
	In this work we focus on hierarchically off-diagonal low-rank matrices (HODLR), which allow
	to store a $N \times N$ matrix of quasiseparable rank $k$
	with $\mathcal O(kN\log N)$ parameters, and to perform
	arithmetic operations (sum, multiplication, inversion)
	in linear-polylogarithmic time.
	
	A representation of a HODLR matrix $A$ is obtained by block
	partitioning it in $2 \times 2$ blocks as follows:
	\begin{equation} \label{eq:HODLRpart} A=
	\begin{bmatrix}
	A_{11}&A_{12}\\
	A_{21}&A_{22}
	\end{bmatrix}
	, \qquad
	A_{11}\in\mathbb C^{N_1\times N_1}, \quad
	A_{22}\in\mathbb C^{N_2\times N_2}
	\end{equation} where $N_1:=\lfloor \frac{N}{2} \rfloor $ and
	$N_2:=\lceil \frac{N}{2} \rceil$. The antidiagonal blocks
	$A_{12}$ and $A_{21}$ have rank at most $k$, and so can be
	efficiently stored as outer products, whereas the
	diagonal blocks $A_{11}$ and $A_{22}$ can be recursively
	represented as HODLR matrices, until we reach a minimal
	dimension. This recursive representation is
	shown in Figure~\ref{fig:hodlr}. In the case of numerically quasiseparable matrices, the off-diagonal blocks are compressed according to the norm of the matrix, i.e., we drop the components of their SVDs, whose magnitude is less than $\epsilon\norm{A}_2$. We will call \emph{truncation tolerance} the parameter $\epsilon$.
	
	\begin{figure}
		\centering
		\includegraphics[width=.8\linewidth]{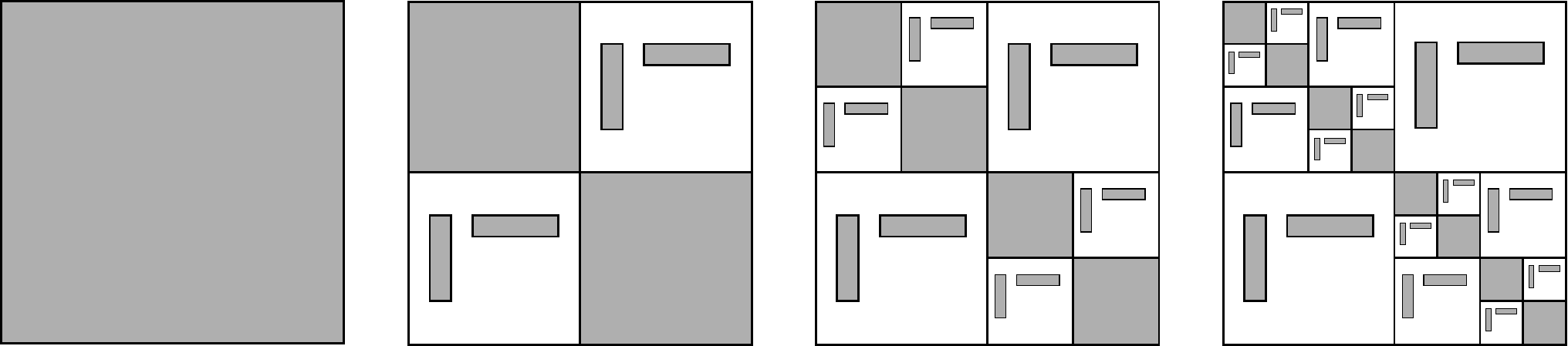}
		\caption{Pictorial representation of the HODLR structure.} \label{fig:hodlr}
	\end{figure}
	In order to efficiently construct the HODLR representation of a matrix, the
	crucial step is to compute a factorized approximation of the
	off-diagonal blocks. We look for procedures whose complexity
	does not exceed the cost of some arithmetic operations in the HODLR format, i.e., matrix-vector multiplication
	(which requires $\mathcal O(kN \log N)$ flops)
	and solving linear systems (requiring $\mathcal O(k^2 N \log^2N)$ flops).
	
	\subsubsection{HODLR representation for finite differences method}
	Our task  is to retrieve a HODLR representation
	of $\fdm = I+ \frac{\Delta t}{\Delta x^\alpha} (D_+^{(m)} T_{\alpha, N} +  D_-^{(m)} T_{\alpha, N}^T)$
	defined in~\eqref{system_FD}. It is easy to see that the tricky
	part is to compress $T_{\alpha,N}$; in fact, performing a diagonal
	scaling is equivalent to scale the left or right factor in the outer
	products, as well as the full diagonal blocks. Finally,
	the shift with the identity only affects the dense diagonal
	blocks.
	
	Assume we have partitioned $T_{\alpha,N}$ as in \eqref{eq:HODLRpart};
	observe that the block $A_{21}$ contains all the subdiagonal blocks
	at the lower levels,
	thanks to the Toeplitz structure. A similar statement holds
	for $A_{12}$.
	Therefore, once we compress $A_{21}$ and $A_{12}$ at
	the top-level, we can obtain all the needed representations
	just by restriction. Moreover, the compression of $A_{12}$ is
	particularly easy, since it has only one nonzero element.
	
	Thus, we have reduced the problem to finding a low-rank representation
	for a Toeplitz matrix (which we know to be numerically low-rank).
	We deal with this issue with the two-sided Lanczos method
	\cite{golub1981block,simon2000low}. This requires
	to perform a few matrix vector products, which can be
	performed, exploiting FFT, in $\mathcal O(kN \log N + Nk^2)$ time.
	
	We remark that the computation of the coefficients
	$g_k^{(\alpha)}$ can be carried out efficiently by recursion
	using the following formulas:
	\[
	g_0^{(\alpha)} = 1,\qquad
	g_1^{(\alpha)} = -\alpha, \qquad
	  g_{k+1}^{(\alpha)} = g_{k}^{(\alpha)} \cdot \left(\frac{k-\alpha}{k+1}\right).
	\]
	
	\subsubsection{HODLR representation for finite elements discretizations}\label{sec:hodlr-fe}
	The proof of Theorem~\ref{thm:structure-riesz-fem} combined
	with Lemma~\ref{lem:alpha-separable} directly provides
	a construction for the low-rank representations of the
	off-diagonal blocks in $\fems$ defined in \eqref{eq:fe-int}.
	At the heart of this machinery,
	it is required to sample the truncated Taylor
	expansions of $(x + y)^{-(\alpha+1)}$. Alternatively,
	separable approximations of this function on the domains
	corresponding to the off-diagonal blocks can be retrieved
	using \texttt{chebfun2} \cite{townsend2013extension}.
	
	However, often the choice of the basis consists
	of shifted copies of the same function on an equispaced
	grid of points. When this is the case, the discretization
	matrix turns out to have Toeplitz plus low-rank structure in the constant
	coefficient case.
    In this situation, the same approximation strategy
	used for finite differences can be applied. This happens in the problem presented in Section~\ref{sec:ex-fe} and considered originally in \cite{zhao}. The authors of the latter proposed to represent the stiffness matrix with a different Hierarchical format. In Figure~\ref{fig:partitioning} we compare the distribution of the off-diagonal ranks obtained with the two partitioning.
	Since there is not a dramatic difference between the ranks obtained with the two choices,
	our approach turns out to be preferable because it provides a lower storage consumption and
	a better complexity for computing the LU decomposition and solving the linear system
	\cite{hackbusch2015hierarchical}.
	
	\begin{figure}

\centering
\begin{tikzpicture}[x = 1.15cm, y = 1.15cm]
\node at (1,1)     {$14$};
\node at (3,3)     {$14$};
\node at (2.5,.5)  {$13$};
\node at (3.5,1.5) {$13$};
\node at (.5,2.5)  {$13$};
\node at (1.5,3.5) {$13$};
\node at (.25, 3.25){$13$};
\node at (.75, 3.75){$13$};
\node at (1.25, 2.25){$13$};
\node at (1.75, 2.75){$13$};
\node at (2.25, 1.25){$13$};
\node at (2.75, 1.75){$13$};
\node at (3.25, .25){$13$};
\node at (3.75, .75){$13$};
\draw (0,0) rectangle (4,4);
\foreach \i in {0, 2} {
	\draw (\i,4-\i) rectangle (\i+2,2-\i);
}
\foreach \i in {0,...,7} {
	\filldraw[lowrankcolor] (\i*.5,4-\i*.5) rectangle (\i*.5+.5,3.5- \i*.5);
}
\foreach \i in {0, ..., 3} {
	\draw (\i,4-\i) rectangle (\i+1,3-\i);
}

\end{tikzpicture}~~~~~~~~~~~~~\begin{tikzpicture}[x = 1.15cm, y = 1.15cm]
\node at (.5,.5)     {$9$};
\node at (1.5,.5)    {$9$};
\node at (.5,1.5)    {$9$};
\node at (3.5,2.5)   {$9$};
\node at (2.5,3.5)   {$9$};
\node at (3.5,3.5)   {$9$};
\node at (.25,2.75)  {$9$};
\node at (.25,2.25)  {$7$};
\node at (.75,2.25)  {$9$};
\node at (1.25,3.75)  {$9$};
\node at (1.75,3.75)  {$7$};
\node at (1.75,3.25)  {$9$};
\node at (1.25,1.75)  {$9$};
\node at (1.25,1.25)  {$7$};
\node at (1.75,1.25)  {$9$};
\node at (2.25,2.75)  {$9$};
\node at (2.75,2.75)  {$7$};
\node at (2.75,2.25)  {$9$};
\node at (2.25,0.75)  {$9$};
\node at (2.25,0.25)  {$7$};
\node at (2.75,0.25)  {$9$};
\node at (3.25,1.75)  {$9$};
\node at (3.75,1.75)  {$7$};
\node at (3.75,1.25)  {$9$};

\foreach \i in {0, 2} {
	\draw (\i,4-\i) rectangle (\i+2,2-\i);
}
\foreach \i in {0,2} {
	\draw (\i,\i+1) -- (\i+2,\i+1);
	\draw (\i+1,\i) -- (\i+1,\i+2);
}
\foreach \i in {0,0.5,...,3} {
	\fill[lowrankcolor] (\i,4-\i) rectangle (\i+1,3-\i);
}
\foreach \i in {0,0.5,...,2.5} {
	\draw (\i,2.5-\i) rectangle (\i+.5,3-\i);
	\draw (\i+1.5,4-\i) rectangle (\i+1,3.5-\i);
}
\draw (0,0) rectangle (4,4);	
\end{tikzpicture}
		\caption{Rank distribution in the HODLR partitioning (left) and in the geometrical clustering (right) for the stiffness matrix of Section~\ref{sec:ex-fe}. The size of the matrix is $N=4096$ and the truncation tolerance is $\epsilon=10^{-8}$.}
		\label{fig:partitioning}
	\end{figure}
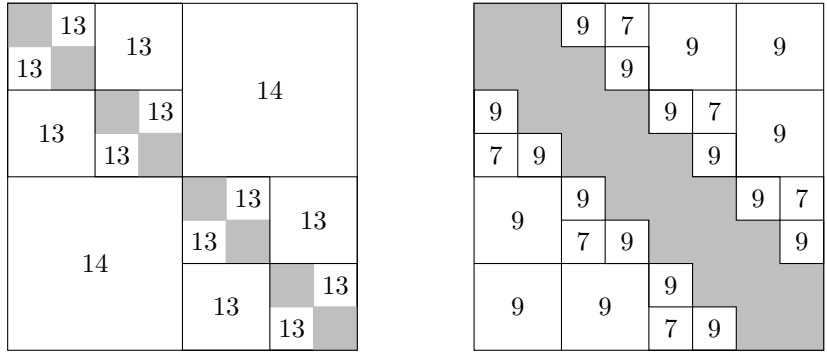
	
	\subsection{Solving the problem in the 1D case}\label{sec:solv-1D}
	As we have shown in the previous section, the 1D discretization
	of the differential equation
	\[
	  \frac{\partial u(x,t)}{\partial t} = d_+(x,t) \frac{\partial^\alpha u(x,t)}{\partial_+ x^\alpha} + d_-(x,t) \frac{\partial^\alpha u(x,t)}{\partial_- x^\alpha} + f(x,t),
	\]
	using either finite differences or finite elements, yields
	a coefficient matrix with HODLR structure, and we have described
	efficient procedures
	to compute its representation. As analyzed in
	\cite{Meer1}, discretizing the
	above equation in time, using implicit Euler,  yields
	\[
	  \frac{u(x,t_{m+1}) - u(x,t_m )}{\Delta t} =
	  d_{+}(x,t_{m+1})\frac{\partial^{\alpha} u(x,t_{m+1})}{\partial_+ x^{\alpha}} +d_{-}(x,t_{m+1})\frac{\partial^{\alpha} u(x,t_{m+1})}{\partial_- x^{\alpha}} +f(x,t_m)+\mathcal O(\Delta t).
	\]
	This leads to a sequence of linear systems
	$\mathcal A u^{(m)} = b^{(m)}$, where $u^{(m)}$
	contains either the samplings on the grid (for finite differences)
	or the coordinates in the chosen basis (for finite elements)
	of $u(x,t_m)$ and  $b^{(m)}$ depends on $f(x,t_m)$ and $u(x,t_{m-1})$.
	The matrix $\mathcal A$ is either $\fdm$ or $\fem$.
	
	To solve the resulting linear system, we rely on fast solvers
	for matrices in the HODLR format, as described in \cite{hackbusch2015hierarchical}.
	In more detail, we first compute a structured LU factorization of the coefficient matrix and then perform back substitution. The quasiseparable property ensures that the off-diagonal rank of the factors of the LU factorization does not differ from the one of the coefficient matrix $\mathcal A$. Computation of LU is the bottleneck and provides an overall complexity of $\mathcal O(k^2N\log^2N)$. Our implementation is based
	on the \texttt{hm-toolbox} \cite{hm-toolbox}.
\subsection{Numerical results for 1D finite differences}
In this section we compare the use of HODLR arithmetic  with a preconditioning technique, recently introduced in \cite{DMS}, for the treatment of certain 1D problems with non constant diffusion coefficients.

We consider the sequence of linear systems $\mathcal A u^{(m)}=b^{(m)}$ arising from the discretization of \eqref{fde_FD} with finite differences. In particular, $\mathcal A$ and $b^{(m)}$ are defined as in \eqref{system_FD}, where we have chosen $\alpha \in\{1.2,1.8\}$, $d_+(x,t)=\Gamma(3-\alpha)x^\alpha$ and $d_-(x,t)=\Gamma(3-\alpha)(2-x)^\alpha$. The spatial domain is $[L,R]=[0,2]$ and we set $\Delta x=\Delta t= \frac 1{N+2}$ for increasing values of $N$. The right-hand side $b^{(m)}\in\mathbb R^{N}$ is chosen as in \cite{DMS}.

In \cite[Example 1]{DMS}, the authors propose two tridiagonal structure preserving preconditioners, defined depending on the value of $\alpha$ and used for speeding up the convergence of GMRES. In particular, such preconditioners ($P_1$ and $P_2$ in the notation of \cite{DMS}) are obtained replacing the Toeplitz matrix $T_{\alpha,N}$ in \eqref{system_FD} with the central difference approximation of the first and the second derivative.

The truncation tolerance has been set to $10^{-8}$ and the dimension of the minimal blocks in the HODLR format is $256$.
The thresholds for the stopping criterion of the GMRES have been set to $10^{-7}$
and $10^{-6}$ in the cases with $\alpha = 1.2$ and $\alpha = 1.8$, respectively,
as this provided comparable accuracies with the HODLR solver.
We compare the time consumption of this method with the one proposed in Section~\ref{sec:solv-1D}, for solving one linear system.
The time for computing the LU factorization and for performing the back substitution are kept separate. In fact, for this example the diffusion coefficients do not depend on time, so in case of multiple time steps the LU can be computed only once at the beginning saving computing resources.

The results reported in Table~\ref{tab:1Da}-\ref{tab:1Db} refer to the performances of the two approaches for a fixed time step (the first one). With \emph{Res} we indicate the relative residue $\norm{\mathcal A x -b^{(m)}}_2/\norm{x}_2$, where $x$ is the computed solution, while \emph{its} denotes the number of iterations needed by the preconditioned GMRES to converge. We note that both strategies scale nicely with respect to the dimension. Once the LU is available, HODLR back substitution provides a significant saving of execution time with respect to preconditioned GMRES. In particular, at dimension $N=\num{131072}$ our approach is faster whenever we need to compute more than about $13$ time steps ($\alpha = 1.2$), or $25$ time steps ($\alpha = 1.8$). The column denoted by
$\text{qsrank}_\epsilon$ indicates the quasiseparable rank of the discretization matrices. 

This example highlights the convenience of the HODLR strategy when used for solving several linear systems endowed with the same coefficient matrix. As shown in the next section, such a benefit is particularly advantageous in the 2D setting when the HODLR format is used in combination with Krylov projection schemes for solving FDE problems rephrased as matrix equations.

 \begin{table}[t]
 	\centering 		
 	\small
 	\begin{filecontents*}{e1.dat}
8192	0.045883	9	1.4472e-08	0.012219	1.259e-09	0.24194	11
16384	0.078481	9	1.2985e-08	0.026405	1.3413e-08	0.58095	11
32768	0.14292	9	1.1088e-08	0.055111	7.8264e-09	1.4899	11
65536	0.30824	10	2.6734e-09	0.11399	7.7719e-09	3.6315	12
1.31072e+05	0.57854	10	2.2317e-09	0.22954	7.9895e-09	9.0552	12
\end{filecontents*}
\pgfplotstabletypeset[%
 	column type=l,
 	every head row/.style={
 		before row={
 			\toprule
 			& \multicolumn{1}{c|}{}
 			& \multicolumn{4}{c|}{PGMRES} & \multicolumn{2}{c}{HODLR}\\
 		},
 		after row = \midrule ,
 	}, 	
 	every last row/.style={after row=\bottomrule},		
 	sci zerofill,
 	columns={0,1,2,3,4,5,6,7},
 	columns/0/.style={dec sep align={c|},column type/.add={}{|},column name=$N$,fixed},
 	columns/1/.style={column name=Time},
 	columns/2/.style={column name=its},
 	columns/3/.style={dec sep align={c|},column type/.add={}{|},column name=$Res$},
 	columns/4/.style={column name=Time},
 	columns/5/.style={column name=Res},
 	columns/6/.style={column name=Time ($LU$)},
 	columns/7/.style={column name=$\text{qsrank}_\epsilon$}
 	]{e1.dat}
 	\caption{Performances of the preconditioned GMRES and of the HODLR solver in the case $\alpha = 1.8$}
 	\label{tab:1Da}
 \end{table}
 \begin{table}[t]
 	\centering 		
 	\small
 	\begin{filecontents*}{e2.dat}
8192	0.047007	9	1.0827e-08	0.012681	2.8685e-09	0.23182	11
16384	0.070328	9	1.0802e-08	0.032696	3.778e-09	0.73604	11
32768	0.14022	9	1.1127e-08	0.065435	5.292e-09	1.8325	11
65536	0.26124	9	1.2038e-08	0.11041	6.885e-09	3.5338	12
1.31072e+05	1.0134	9	1.2975e-08	0.22676	7.9228e-09	10.47	12
\end{filecontents*}
\pgfplotstabletypeset[%
 	column type=l,
 	every head row/.style={
 		before row={
 			\toprule
 			& \multicolumn{1}{c|}{}
 			& \multicolumn{4}{c|}{ PGMRES} & \multicolumn{2}{c}{ HODLR}\\
 		},
 		after row = \midrule ,
 	}, 	
 	every last row/.style={after row=\bottomrule},		
 	sci zerofill,
 	columns={0,1,2,3,4,5,6,7},
 	columns/0/.style={column name=$N$,fixed,dec sep align={c|},column type/.add={}{|}},
 	columns/1/.style={column name=Time},
 	columns/2/.style={column name=its},
 	columns/3/.style={dec sep align={c|},column type/.add={}{|},column name=Res},
 	columns/4/.style={column name=Time},
 	columns/5/.style={column name=Res},
 	columns/6/.style={column name=Time ($LU$)},
 	columns/7/.style={column name=$\text{qsrank}_\epsilon$}
 	]{e2.dat}
 	\caption{Performances of the preconditioned GMRES and of the HODLR solver in the case $\alpha = 1.2$}
 	\label{tab:1Db}
 \end{table}

	\section{Spatial 2D problems with piece-wise smooth right-hand side}\label{sec:matrix-eq}
	
	We now describe how to efficiently
	solve 2D fractional differential equations
	leveraging the rank properties that we have identified in the
	1D discretizations.
	
	More precisely, we are interested in solving the
	equation
	\begin{align}\label{eq:2dmatrix}
	  \frac{\partial u}{\partial t} &= d_{1,+}(x,t)\frac{\partial^{\alpha_1} u}{\partial_+ x^{\alpha_1}} +d_{1,-}(x,t)\frac{\partial^{\alpha_1} u}{\partial_- x^{\alpha_1}} + d_{2,+}(y,t)\frac{\partial^{\alpha_2} u}{\partial_+ y^{\alpha_2}} +d_{2,-}(y,t)\frac{\partial^{\alpha_2} u}{\partial_- y^{\alpha_2}}+f,
	\end{align}
	where $(x,y) \in [a, b] \times [c, d]$, $t\geq 0$
	and imposing absorbing boundary conditions. 
	
	 We discretize \eqref{eq:2dmatrix}
	in the time variable using implicit Euler, and we obtain
	\begin{align} \label{eq:2dfde}
	\begin{split}
	  \frac{u(x,y,t_{m+1}) - u(x,y,t_m )}{\Delta t} &=
	    d_{1,+}(x,t_{m+1})\frac{\partial^{\alpha_1} u(x,y,t_{m+1})}{\partial_+ x^{\alpha_1}} +d_{1,-}(x,t_{m+1})\frac{\partial^{\alpha_1} u(x,y,t_{m+1})}{\partial_- x^{\alpha_1}} \\ &+ d_{2,+}(y,t_{m+1})\frac{\partial^{\alpha_2} u(x,y,t_{m+1})}{\partial_+ y^{\alpha_2}} +d_{2,-}(y,_{m+1})\frac{\partial^{\alpha_2} u(x,y,t_{m+1})}{\partial_- y^{\alpha_2}} \\ &+f(x,y,t_{m+1})
	    + \mathcal O(\Delta t).
	\end{split}
	\end{align}
Then, we discretize the space derivative
	by considering a tensorized form of a 1D discretization. In
	the finite differences case, we consider a grid of nodes
	obtained as the product of equispaced points on $[a,b]$ and
	$[c,d]$, respectively. In the finite elements case, we
	assume that the basis is formed using products of basis
	elements in the coordinates $x$ and $y$.
This leads to linear systems of the form
\[
\mathcal A \ \vect(U^{(m+1)})=\vect(U^{(m)})+\vect(F^{(m+1)}),
\]
	where we used the operator $\vect(\cdot)$ because, as we will discuss later, it is useful to reshape these objects into matrices conformally to the discretization grid.
	
	\begin{remark}\label{rem:diff}
		In the formulation of \eqref{eq:2dmatrix}, the diffusion coefficients multiplying the differential operator only
		depend on time and on the variable involved in the differentiation. This is
		not by accident, since it makes it faster to rephrase the
		problem in matrix equation form as we will do
		in Section~\ref{sec:rephrasing-matrix-equations}.
		Indeed, these assumptions
		are also present in \cite{Breiten2014}, where the connection
		with matrix equations has been introduced for fractional
		differential equations.
	\end{remark}
	
	\subsection{Regularity and rank structure in the right-hand side}
	When the source term $f(x,y,t)$ is smooth in the spatial variables
	at every time step, the matrix $F^{(m+1)}$ turns out to be
	numerically low-rank. This can be justified in several ways.
	For instance, one
	can consider the truncated expansion of
	$f(x,y,t_{m+1})$ in any of the
	spatial variables, similarly to what is done in the proof of
	Lemma~\ref{lem:alpha-separable}. This provides a separable approximation
	of the function, which corresponds to a low-rank approximation
	of $F^{(m+1)}$. Another interesting point of view
	relies on introducing SVD for bivariate functions
	\cite{townsend2014computing}.
	
	In practice, one can recover a low-rank approximation of $F^{(m+1)}$
	by performing a bivariate polynomial expansion, for instance using
	\texttt{chebfun2} \cite{townsend2013extension}. For elliptic differential
	equations, the solution often inherits the same regularity of the right-hand side,
	and therefore the same low-rank
	properties are found in $U^{(m)}$ at every timestep.	
	Once low-rank representations of $F^{(m+1)}$ and
	$U^{(m)}$ are available, it is advisable to recompress
	their sum.
	If the latter has rank $k \ll N$,
	this can be performed in $\mathcal O(Nk^2)$ flops, relying
	on reduced QR factorizations of the outer factors, followed
	by an SVD of a $k \times k$ matrix (see
	the recompression procedure in \cite[Algorithm 2.17]{hackbusch2015hierarchical}).
	
	The same machinery applies when $f$ is piecewise smooth, by decomposing it
	into a short sum $f_1+\dots+f_s$, where each $f_j$ is smooth in a box that contains its support.

	\subsection{Linear systems as matrix equations}
	\label{sec:rephrasing-matrix-equations}
We consider equations as in \eqref{eq:2dmatrix}. Note that the
two differential operators in $x$ and $y$ act independently on
the two variables. Because of this and exploiting the assumption on the diffusion coefficients highlighted in Remark \ref{rem:diff}, the matrix $\mathcal A$ can be written
either as 
\[
I\otimes\left(\frac{1}{2}I-\Delta t \fds[\alpha_1] \right)+  \left(\frac{1}{2}I-\Delta t\fds[\alpha_2] \right)\otimes I, 
\]
in the finite difference case, or as 
\[
  M\otimes\left( \frac{1}{2} M-\Delta t \fems[\alpha_1] \right)+ 
    \left(\frac{1}{2}M-\Delta t  \fems[\alpha_2] \right)\otimes M, 
\]
in the finite element case, where $M$ is the 1D mass matrix. 
Using the well-known relation $\vect(AXB)=(B^T\otimes A)\vect(X)$ we
		get the Sylvester equation
		\begin{equation} \label{eq:sylv-timestepping}
		\left(\frac 12 I -\Delta t\fds[\alpha_1]\right)
		U^{(m+1)} + U^{(m+1)} \left(\frac 12 I -\Delta t\fds[\alpha_2]\right)^T = F^{(m+1)} + U^{(m)},
		\end{equation}
		for the finite difference case,
		where $U^{(m)}$ is the solution approximated at the time step
		$m$, and $F^{(m)}$ contains the sampling of the function $f(x,y,t_{m+1})$ on the grid. In the case of finite elements, instead, one obtains the generalized Sylvester equation
		\begin{equation} \label{eq:sylv-timestepping-fe}
	\left(\frac 12 M -\Delta t\fems[\alpha_1]\right)
	U^{(m+1)}M + MU^{(m+1)} \left(\frac 12 M -\Delta t\fems[\alpha_2]\right)^T = M F^{(m+1)}M + MU^{(m)}M.
		\end{equation}
		We can obtain the same structure of 
        \eqref{eq:sylv-timestepping} by inverting $M$, if it is well-conditioned, or treat the problem directly considering the pencils $\frac 12 M -\Delta t\fems[\alpha_1] - \lambda M$
        and $\frac 12 M -\Delta t\fems[\alpha_2] - \lambda M$ \cite{simoncini2016computational}. In the experiment of Section~\ref{sec:ex-fe}, we rely
        on the first approach. 
		
	In light of the properties of $F^{(m+1)}$ and $U^{(m)}$, we have reformulated a space discretization of \eqref{eq:2dfde}
as a matrix equation of the form
	\[
	  AX + XB = UV^T, \qquad
	  U,V \in \mathbb C^{N \times k},
	\]
	where $A$ and $B$ are square and $k\ll N$. From now on, we assume the spectra of $A$ and $-B$ to be separated by a line. This ensures that the sought solution has numerical low-rank \cite{Beckermann2016}.

    \subsection{Fast methods for linear matrix equations}
    Linear matrix equations are well-studied since they arise in
    several areas, from control theory to PDEs. In our case the
    right-hand side is low-rank, and the structure in
    the matrices $A$ and $B$ allows to perform
    fast matrix vector multiplications and system solutions. For this
    reason, we choose to apply the \emph{extended Krylov subspace
    method} introduced in \cite{simoncini2007new}.

    This procedure constructs  orthonormal bases $U_s$
    and $V_s$ for the
    subspaces
	\begin{align*}
	  \mathcal{EK}_s(A, U) &= \mathrm{span}\{ U, A^{-1}U, AU, \ldots,
	    A^{s-1}U, A^{1-s}U \} \\
	  \mathcal{EK}_s(B^T, V) &= \mathrm{span}\{ V, B^{-T}V, B^T V, \ldots,
	    (B^T)^{s-1}V, (B^T)^{1-s}V \},
	\end{align*}
    by means of two extended block Arnoldi processes \cite{heyouni2010extended}.
    Then, the compressed
    equation $\tilde A_s X_s + X_s \tilde B_s = \tilde U \tilde V^T$
    is solved,  where
    $\tilde A_s = U_s^* A U_s$, $\tilde B_s = V_s^* B V_s$,
    $\tilde U = U_s^* U$, and $\tilde V = V_s^* V$. The latter
    equation is small scale ($s \times s$, with
    $s \ll n$), and can be solved
    using dense linear algebra.
    An approximation of the solution is finally provided by $U_s X_sV_s^*$.

    The complexity of the procedure depends on the convergence of the extended Krylov subspace method, which is related to the spectral properties of $A$ and $B$ \cite{beckermann2011error,Knizhnerman2011}. Under the simplified assumption that the Krylov method converges in a constant number of iterations, the overall complexity is determined by the precomputation of the LU factorization of $A$ and $B$, i.e. $\mathcal O(N\log^2N)$.

    A robust and efficient implementation of this technique requires
    some care, especially in the case where the rank in the right-hand side is larger than
    $1$. We refer to \cite{gutknecht2005block,simoncini2007new} for an overview of the
    numerical issues.
	
	\section{Numerical results and comparisons}\label{sec:exp}
	
	All the results in this section have been run on 
	 MATLAB R2017a, using a laptop with multithreaded Intel MKL BLAS
	and a i7-920 CPU with 18GB of RAM. The implementation of the
	fast HODLR arithmetic and the extended Krylov method can be found in \texttt{hm-toolbox} \cite{hm-toolbox}. The
	block Arnoldi process is taken from
	\texttt{rktoolbox} \cite{berljafa2015generalized}.
	
	The codes used for the tests are available at
	\url{https://github.com/numpi/fme}. They are organized
	as MATLAB m-files to allow replicating the results in this sections. We
	thank the authors of \cite{Breiten2014} who made their code public,
	allowing us to easily reproduce their results.

	\subsection{2D time-dependent equation with finite difference scheme}
	
	In this section we compare the use of rank-structured arithmetic
	embedded in the extended Krylov solver with the use of an
	appropriately preconditioned GMRES as proposed in \cite{Breiten2014}
	by Breiten, Simoncini, and Stoll. For the sake of simplicity, from now on
	we
	refer to the former method with the shorthand HODLR, and to the latter
	as BSS. This notation is also used in figures and tables where the performances
	are compared.
	
	The test problem is taken directly from \cite{Breiten2014}. The equation
	under consideration is \eqref{eq:2dmatrix} 
	with absorbing boundary conditions. The spatial domain is the square
	$[0, 1]^2$, and the source term $f$ is chosen as follows:
	\[
	  f(x,y,t) = 100 \cdot \left(\sin(10 \pi x) \cos(\pi y) +
	    \sin(10 t) \sin(\pi x) \cdot y (1 - y) \right).
	\]
	We consider two instances of this problem. The first one is a
	constant coefficient case, i.e., the
	diffusion coefficients ($d_{1,\pm}$ and $d_{2,\pm}$) are all equal to the
	constant $1$. In the second, instead, we choose them as follows:
    \begin{align}\label{eq:coeff_es}
   \begin{split}
      d_{1,+}(x) &= \Gamma(1.2) (1 + x)^{\alpha_1}, \qquad
      d_{1,-}(x) = \Gamma(1.2) (2 - x)^{\alpha_1},  \\
      d_{2,+}(y) &= \Gamma(1.2) (1 + y)^{\alpha_2}, \qquad
      d_{2,-}(y) = \Gamma(1.2) (2 - y)^{\alpha_2}.
    \end{split}
    \end{align}
    According to our discussion in Section~\ref{sec:matrix-eq},
    we know how to recast the space-time discretization
    in matrix equation form. More precisely, we consider the
    implicit Euler scheme in time with $\Delta t = 1$, and the Gr\"unwald-Letnikov
    shifted finite difference scheme for the space discretization, with
    a space step $\Delta x= \Delta y = \frac{1}{N+2}$. This yields a time stepping scheme
    that requires the solution
    of a Sylvester equation in the form \eqref{eq:sylv-timestepping} at each step.
    In particular, we
    note that the sampling of $f(x,y,t)$ on the discretization grid is of
    rank (at most) $2$ independently of the time.
    We performed $8$ time steps, coherently with the setup for the
    experiments used in \cite{Breiten2014}.

    The timings of the two approaches for the constant coefficient case
    are reported in
    Table~\ref{fig:results_breiten_1}
    for $\alpha_1 = 1.3, \alpha_2 = 1.7$, and
    in Table~\ref{fig:results_breiten_2} for $\alpha_1 = 1.7, \alpha_2 = 1.9$. The
    same tests in the non-constant coefficients setting have been performed,
    and the results are reported in Table~\ref{fig:results_breiten3} and Table~\ref{fig:results_breiten4}.

    The stopping criterion for the extended Krylov method has been set
    to $\epsilon := 10^{-6}$; this guarantees that the residual of the linear
    system will be smaller than $\epsilon$. The stopping criterion with the
    relative residual  for GMRES
    has been chosen as $10^{-7}$ 
    and the truncation tolerance for the operation in HODLR arithmetic (only used
    when assembling the matrices) to $10^{-8}$. 

    \begin{table}
    	\centering
    	\begin{minipage}{.5\linewidth} \centering
    		\begin{filecontents*}{fd-times_13.dat}
512	0.45018	15	11	512	3.4336	15	0
1024	0.77233	15	11	1024	5.8126	15	0
2048	1.3449	15	12	2048	10.378	15	0
4096	2.4437	16	12	4096	28.178	16	0
8192	4.2766	16	13	8192	45.728	16	0
16384	8.5324	16	13	16384	89.804	16	0
32768	19.319	16	14	32768	196.72	16	0
65536	44.303	16	14	65536	434.82	16	0
\end{filecontents*}
\pgfplotstabletypeset[
    		columns={0,1,5,2,3},
    		columns/0/.style = {column name = $N$},
    		columns/1/.style = {column name = $t_{\mathrm{HODLR}}$},
    		columns/2/.style = {column name = $\mathrm{rank}_{\epsilon}$},
    		columns/5/.style = {column name = $t_{\mathrm{BSS}}$},
    		columns/3/.style={column name=$\text{qsrank}_\epsilon$}
    		]{fd-times_13.dat}
    	\end{minipage}~\begin{minipage}{.45\linewidth} \centering
    	\bigskip
    	\begin{tikzpicture}
    	\begin{loglogaxis}[width = .95\linewidth, legend pos = north west,
    	  xlabel = $N$, ylabel = Time (s)]
    	\addplot table[x index = 0, y index = 1] {fd-times_13.dat};
    	\addplot table[x index = 0, y index = 5] {fd-times_13.dat};
    	\legend{HODLR, BSS}
    	\end{loglogaxis}
    	\end{tikzpicture}
    \end{minipage}
    	\caption{Timings for the solution of the problem \eqref{eq:2dmatrix}
    		in the constant coefficient case using the HODLR
    		solver and the approach presented in \cite{Breiten2014}.
    		The exponents are set to $\alpha_1 = 1.3$ and $\alpha_2 = 1.7$.
    		The times reported are expressed in seconds. }
    	\label{fig:results_breiten_1}
    \end{table}

    In Figure~\ref{fig:solution-prob1} we report a plot of
    the final solution of the constant coefficient problem
    at time $7$; the parameters in the figure are
    $\alpha_1 = 1.3$ and $\alpha_2 = 1.7$. The field $\text{rank}_\epsilon$
    indicates the numerical rank of the solution, whereas $\text{qsrank}_\epsilon$
    denotes the numerical quasiseparable rank of the discretization matrices. The
    latter increase proportionally to $\log(N)$, in this and
    in the following examples, as predicted by the theory. 

    \begin{figure}
    	\centering
    	\includegraphics[width=.49\linewidth]{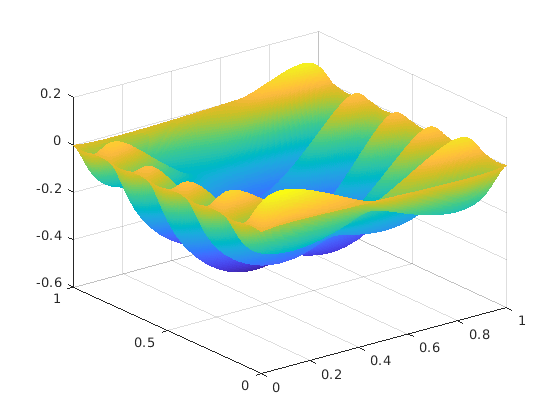}~\includegraphics[width=.49\linewidth]{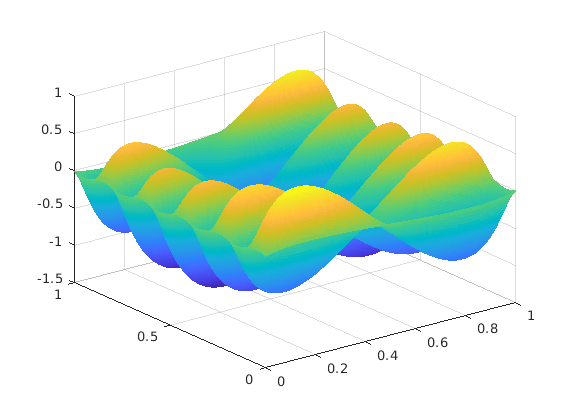}
    	\caption{Solutions of the problem \eqref{eq:2dmatrix} in the constant
    		coefficient case. On the left, the
    		problem with $\alpha_1 = 1.7, \alpha_2 = 1.9$. On the right,
    		the one with $\alpha_1 = 1.3, \alpha_2 = 1.7$. The solutions
    		are plotted at the final time step $t = 7$.}
    	\label{fig:solution-prob1}
    \end{figure}

        \begin{table}
        	\centering
        	\begin{minipage}{.5\linewidth} \centering
        		\begin{filecontents*}{fd-times_24.dat}
512	0.63436	18	10	512	2.5293	18	0
1024	1.0547	19	10	1024	4.2137	18	0
2048	1.9665	20	11	2048	7.5023	19	0
4096	4.1943	21	11	4096	18.339	20	0
8192	7.2448	21	11	8192	31.9	20	0
16384	17.053	21	11	16384	55.858	21	0
32768	32.477	22	11	32768	107.85	22	0
65536	79.245	23	11	65536	221.83	23	0
\end{filecontents*}
\pgfplotstabletypeset[
        		columns={0,1,5,2,3},
        		columns/0/.style = {column name = $N$},
        		columns/1/.style = {column name = $t_{\mathrm{HODLR}}$},
        		columns/2/.style = {column name = $\mathrm{rank}_{\epsilon}$},
        		columns/5/.style = {column name = $t_{\mathrm{BSS}}$},
        		columns/3/.style={column name=$\text{qsrank}_\epsilon$}
        		]{fd-times_24.dat}
        	\end{minipage}~\begin{minipage}{.45\linewidth} \centering
    	\bigskip
    	\begin{tikzpicture}
    	\begin{loglogaxis}[width = .95\linewidth, legend pos = north west,
    	xlabel = $N$, ylabel = Time (s)]
        		  \addplot table[x index = 0, y index = 1] {fd-times_24.dat};
        		  \addplot table[x index = 0, y index = 5] {fd-times_24.dat};
        		  \legend{HODLR, BSS}
        		\end{loglogaxis}
        	\end{tikzpicture}
        \end{minipage}
        \caption{Timings for the solution of the problem \eqref{eq:2dmatrix}
        	with constant coefficients as in \eqref{eq:coeff_es} using the HODLR
        	solver and the approach presented in \cite{Breiten2014}.
        	The exponents are set to $\alpha_1 = 1.7$ and $\alpha_2 = 1.9$. }
        \label{fig:results_breiten_2}
    \end{table}

    We note that the HODLR solver outperforms the BSS approach in all our tests,
    although the advantage is slightly reduced when $N$ increases. For solving Toeplitz linear systems,
    we have used the preconditioner in \cite{DMS}, which turned out to perform better than
    the circulant one.
    However,
    when considering the the nonconstant coefficients case,
    a small growth in the number of iterations can be seen as $N$ increases. In
    particular, the preconditioner based on the first derivative $P_1$ seems to be less robust, while $P_2$ is less
    sensitive to changes in the coefficients, the time-step,
    and other parameters. In this example,
    it turned out that $P_2$ is always the most efficient choice, even when
    $\alpha_1$ is as low as $1.3$, and therefore it is the one employed in the
    tests of the BSS method.

    Note that in the case with $\alpha_1 = 1.3$
    the preconditioner $P_2$ works
    well (the number of iterations does not increase with $N$), but the number of
    iterations is not particularly low (typical figures are in the range of 15 to 20,
    sometimes more);
    therefore, this case
    is particularly favorable to the HODLR approach, which indeed outperforms
    BSS by a factor of about $6$ in time at the
    larger tested size, $N = \num{65536}$. The results are reported
    in Table~\ref{fig:results_breiten_1} and Table~\ref{fig:results_breiten3}.

    \begin{table}
    	\centering
    	\begin{minipage}{.5\linewidth} \centering
    		\begin{filecontents*}{fd-times-vc_13.dat}
512	0.33843	14	10	512	3.8386	16	0
1024	0.61323	14	11	1024	6.4257	17	0
2048	1.1394	16	12	2048	11.683	17	0
4096	2.5932	16	12	4096	30.153	18	0
8192	4.8634	16	13	8192	51.186	18	0
16384	8.7924	17	13	16384	84.682	18	0
32768	19.235	15	14	32768	153.98	19	0
65536	46.127	15	14	65536	282.35	18	0
\end{filecontents*}
\pgfplotstabletypeset[
    		columns={0,1,5,2,3},
    		columns/0/.style = {column name = $N$},
    		columns/1/.style = {column name = $t_{\mathrm{HODLR}}$},
    		columns/2/.style = {column name = $\mathrm{rank}_{\epsilon}$},
    		columns/5/.style = {column name = $t_{\mathrm{BSS}}$},
    		columns/3/.style={column name=$\text{qsrank}_\epsilon$}
    		]{fd-times-vc_13.dat}
    	\end{minipage}~\begin{minipage}{.45\linewidth} \centering
    	\bigskip
    	\begin{tikzpicture}
    	\begin{loglogaxis}[width = .95\linewidth, legend pos = north west,
    	xlabel = $N$, ylabel = Time (s)]
    	\addplot table[x index = 0, y index = 1] {fd-times-vc_13.dat};
    	\addplot table[x index = 0, y index = 5] {fd-times-vc_13.dat};
    	\legend{HODLR, BSS}
    	\end{loglogaxis}
    	\end{tikzpicture}
    \end{minipage}
    \caption{Timings for the solution of the problem \eqref{eq:2dmatrix}
    	with variable coefficients as in \eqref{eq:coeff_es} using the HODLR
    	solver and the approach presented in \cite{Breiten2014}.
    	The exponents are set to $\alpha_1 = 1.3$ and $\alpha_2 = 1.7$.
    	The times reported are expressed in seconds. }
    \label{fig:results_breiten3}
\end{table}

\begin{table}
	\centering
	\begin{minipage}{.5\linewidth} \centering
		\begin{filecontents*}{fd-times-vc_24.dat}
512	0.59936	18	10	512	2.6839	20	0
1024	1.009	19	10	1024	4.356	20	0
2048	2.0347	21	11	2048	7.6914	21	0
4096	4.0563	22	11	4096	16.854	21	0
8192	8.5249	22	11	8192	28.272	21	0
16384	17.729	22	11	16384	48.642	22	0
32768	36.817	24	12	32768	88.814	21	0
65536	95.431	25	12	65536	170.94	22	0
\end{filecontents*}
\pgfplotstabletypeset[
		columns={0,1,5,2,3},
		columns/0/.style = {column name = $N$},
		columns/1/.style = {column name = $t_{\mathrm{HODLR}}$},
		columns/2/.style = {column name = $\mathrm{rank}_{\epsilon}$},
		columns/5/.style = {column name = $t_{\mathrm{BSS}}$},
		columns/3/.style={column name=$\text{qsrank}_\epsilon$}
		]{fd-times-vc_24.dat}
	\end{minipage}~\begin{minipage}{.45\linewidth} \centering
    	\bigskip
    	\begin{tikzpicture}
    	\begin{loglogaxis}[width = .95\linewidth, legend pos = north west,
    	xlabel = $N$, ylabel = Time (s)]
	\addplot table[x index = 0, y index = 1] {fd-times-vc_24.dat};
	\addplot table[x index = 0, y index = 5] {fd-times-vc_24.dat};
	\legend{HODLR, BSS}
	\end{loglogaxis}
	\end{tikzpicture}
\end{minipage}
\caption{Timings for the solution of the problem \eqref{eq:2dmatrix}
	with variable coefficients as in \eqref{eq:coeff_es} using the HODLR
	solver and the approach presented in \cite{Breiten2014}.
	The exponents are set to $\alpha_1 = 1.7$ and $\alpha_2 = 1.9$. }
\label{fig:results_breiten4}
\end{table}

\subsection{A 2D finite element discretization}\label{sec:ex-fe}

The example considered in this section is given by problem \eqref{eq:2dmatrix} on the domain $[0, 1]^2$ in the
constant coefficients case and with source term, and solution at time $0$ defined
as follows:
\[
  f(x,y,t) = \begin{cases}
   1 & (x,y) \in H \\
   0 & \text{otherwise} \\
  \end{cases}, \qquad
  u(x,y,0) = f(x,y,0), \qquad
  H = \left\{ \frac{3}{8} \leq x, y\leq \frac{5}{8} \right\},
\]
similarly to what is done in \cite{duan2015finite}. The fractional
exponents are chosen as $\alpha_1 = 1.3$ and $\alpha_2 = 1.7$.
We used the piecewise linear functions described in \cite{duan2015finite} as basis for the finite element discretization.
For simplicity, we consider a uniform grid for the nodes defining
the hat functions, which yields
a Toeplitz matrix whose symbol is explicitly\footnote{
	The formula given in \cite{lin2017finite} is not numerically stable,
	and gives rise to severe cancellation errors if used to compute
	element far from the diagonal. However, it can be easily stabilized
	performing a series expansion of the terms involved, and
	removing the terms that are known to cancel out; this yields
	an expression as a convergent series; the latter can be
	efficiently evaluated by truncation, and this is what we
	have done in our implementation.} given in \cite{lin2017finite}.
Therefore, the same strategy used in the previous section can be used
to recover a rank structured representation of the matrix,
which is guaranteed to exist thanks to Theorem~\ref{thm:structure-riesz-fem}.

We consider the time step of $\Delta t = 0.1$, and the discretization in
time is done by the backward Euler method. The truncation thresholds
are set exactly as in the previous example. We have performed tests
changing the number of grid points used in each direction, and the
timings are reported in Table~\ref{fig:results_fe_1}.

\begin{table}
	\centering
	\begin{minipage}{.5\linewidth} \centering
		\begin{filecontents*}{fe-times_13.dat}
512	0.21957	10	13	512	1.4474	10	0
1024	0.42022	11	14	1024	2.0073	11	0
2048	0.7966	12	13	2048	5.5248	12	0
4096	1.8017	13	13	4096	7.9122	12	0
8192	3.0492	13	14	8192	22.53	13	0
16384	7.1786	14	14	16384	42.353	13	0
32768	12.884	15	15	32768	84.935	14	0
65536	31.659	14	15	65536	122.62	14	0
\end{filecontents*}
\pgfplotstabletypeset[
		columns={0,1,5,2,3},
		columns/0/.style = {column name = $N$},
		columns/1/.style = {column name = $t_{\mathrm{HODLR}}$},
		columns/2/.style = {column name = $\mathrm{rank}_{\epsilon}$},
		columns/5/.style = {column name = $t_{\mathrm{BSS}}$},
		columns/3/.style={column name=$\text{qsrank}_\epsilon$}
		]{fe-times_13.dat}
	\end{minipage}~\begin{minipage}{.45\linewidth} \centering
	\bigskip
	\begin{tikzpicture}
	\begin{loglogaxis}[width = .95\linewidth, legend pos = north west,
	xlabel = $N$, ylabel = Time (s)]
	\addplot table[x index = 0, y index = 1] {fe-times_13.dat};
	\addplot table[x index = 0, y index = 5] {fe-times_13.dat};
	\legend{HODLR, BSS}
	\end{loglogaxis}
	\end{tikzpicture}
\end{minipage}
\caption{Timings for the solution of the problem \eqref{eq:2dmatrix}
	using a finite element discretization.
	The exponents are set to $\alpha_1 = 1.3$ and $\alpha_2 = 1.7$. }
\label{fig:results_fe_1}
\end{table}

We notice that the timings behave linearly, and the rank
of the solution stabilizes around $14$. Since the matrix equation
has Toeplitz coefficients, we can directly compare with the
approach by Breiten, Simoncini, and Stoll in \cite{Breiten2014}. A spectral reasoning similar to the one in \cite{DMS} justifies the use of preconditioner $P_2$ for solving Toeplitz linear systems also in the finite elements setting. As for the finite
differences, the use of rank structures turns out to be more efficient. However,
we stress that our method remains applicable without modifications
even if one considers non-uniform grids, in contrast to the
BSS approach. The only difference is in
the construction of the rank structured representation, which needs
to be performed relying on Theorem~\ref{thm:structure-riesz-fem}.

\section{Conclusions and outlook}\label{sec:concl}

In this paper we have presented a rigorous theoretical analysis of the rank
of the off-diagonal blocks in 1D discretizations of fractional differential operators.
We have analyzed different formulations, namely the Gr\"unwald-Letnikov (shifted and
non-shifted) finite difference schemes, as well as finite element approaches.
In the latter class, we have shown that the stiffness matrix of the finite element
discretization is rank structured under  mild hypotheses on the finite element
basis.

We have then shown that it is possible to obtain parametrizations of such rank
structures very efficiently in the HODLR format,
and this can be used to solve 1D fractional
differential equations (possibly with time dependence), as well as to analyze a
broad range of 2D problems where the differential operator is separable,
and thus the equation can be recast in matrix form \cite{Breiten2014}.
This includes, but is not limited to, fractional diffusion equations.

In our numerical experiments we have shown that HODLR-based solvers
often outperform previous approaches relying on Toeplitz-structured preconditioners.
 This is particularly advantageous in the 2D setting, where
in the projection scheme used to deal with the matrix equation one needs to solve several
linear systems, and the computation of the LU factorization in HODLR format
can be amortized among more operations. 

The machinery extends to 2D equations whose associated matrix equation has a right-hand side in the HODLR format.  In this case, it is necessary to replace the extended Krylov method with the divide and conquer technique presented in \cite{kressner2017low}.

Further improvements can be achieved replacing the HODLR format with more sophisticated structures that rely on nested bases for the representation of the off-diagonal blocks, as HSS and $\mathcal H^2$ matrices \cite{chandra,hackbusch2015hierarchical}. This would remove some log factors from the asymptotic complexity of time and memory consumption, and
might be subject of future work.

\section*{Acknowledgment}

The authors wish to thank the CIRM (Centre International de Rencontres Mathématiques) in Luminy, France, which supported
a ``Research in Pairs'' on the
topic of fast methods for fractional differential equations.
Part of the work presented in this paper is a result of that meeting.

\bibliographystyle{abbrv}
\bibliography{library}
\end{document}